\documentclass[review,onefignum,onetabnum]{siamart220329}

\usepackage{graphics} 
  \usepackage{epsfig}
\usepackage{lipsum}
\usepackage{amsfonts}
\usepackage{amssymb}
\usepackage{graphicx}
\usepackage{epstopdf}
\usepackage{algorithm}
\usepackage{algorithmic}

\usepackage{subfigure}
\ifpdf
  \DeclareGraphicsExtensions{.eps,.pdf,.png,.jpg}
\else
  \DeclareGraphicsExtensions{.eps}
\fi

\usepackage{enumitem}
\setlist[enumerate]{leftmargin=.5in}
\setlist[itemize]{leftmargin=.5in}

\newsiamremark{remark}{Remark}
\newsiamremark{hypothesis}{Hypothesis}
\newsiamremark{assumption}{Assumption}
\crefname{hypothesis}{Hypothesis}{Hypotheses}
\newsiamthm{claim}{Claim}

\ifpdf
\hypersetup{
  pdftitle={An Optimal Control Problem with uncertain terminal complementarity constraints},
  pdfauthor={J. Luo and X. Chen}
}
\fi

\headers{ Dynamic Systems with Stochastic Nonsmooth Convex Optimization }{J. LUO and X. CHEN}

\title{ Dynamic Systems Coupled with Solutions of Stochastic Nonsmooth Convex Optimization \thanks{1 September 2023, revisions 9 March 2024, 14 August 2024, 24 January 2025,
\funding{This work is supported by CAS-Croucher Funding Scheme for  the CAS AMSS-PolyU Joint Laboratory in Applied Mathematics and Hong Kong Research Grant Council project PolyU15300022.}}}

\author{Jianfeng Luo\thanks{School of Mathematics, North University of China, Taiyuan, China;
  CAS AMSS-PolyU Joint Laboratory of Applied Mathematics (Shenzhen), The Hong Kong Polytechnic University Shenzhen Research Institute, Shenzhen, China
  (\email{luojf2024@163.com}).}
\and Xiaojun Chen\thanks{Department of Applied Mathematics, The Hong Kong Polytechnic University, Kowloon, Hong Kong (\email{maxjchen@polyu.edu.hk}).}}

\usepackage{amsopn}

\begin{document}

\maketitle

\begin{abstract}
In this paper, we study ordinary differential equations (ODE) coupled with solutions of a stochastic nonsmooth convex optimization problem (SNCOP). We use the regularization approach, the sample average approximation and the time-stepping method to construct discrete approximation problems. We show the existence of solutions to the original problem and the discrete problems. Moreover, we show that the optimal solution of the SNCOP with a strong convex objective function admits a linear growth condition and the optimal solution of the regularized SNCOP converges to the least-norm solution of the original SNCOP, which are crucial for us to derive the convergence results of the discrete problems.  We illustrate the theoretical results and applications for the estimation of the time-varying parameters in ODE by numerical examples.
\end{abstract}

\vspace{-0.05in}

\vspace{-0.05in}
\vspace{-0.05in}
\section{Introduction}
Let $\xi$ be a random variable defined in the probability space $(\Omega,\mathcal{F},\mathcal{P})$ with support set $\Xi:=\xi(\Omega)\subseteq\mathbb{R}^d$. Let $f: \mathbb{R} \times \mathbb{R}^n\times \mathbb{R}^m \times \Xi \rightarrow\mathbb{R}^n$, $g: \mathbb{R} \times \mathbb{R}^n \times \mathbb{R}^m\times \Xi\rightarrow \mathbb{R}$, $A: \Xi \rightarrow \mathbb{R}^{q\times n}$, $B:  \Xi \rightarrow \mathbb{R}^{q\times m}$ and $Q: \mathbb{R} \times \Xi\rightarrow \mathbb{R}^q$ be given mappings.
In this paper, we consider the following dynamic system coupled with solutions of stochastic nonsmooth convex optimization:
\begin{eqnarray}
& &\dot{x}(t)=\mathbb{E}[f(t,x(t),y(t),\xi)],\,\,\, x(0)=x_0, \label{OCDE-1}\\
& &\begin{aligned}
&y(t)\in\arg\min_{\mathbf{y}\in \mathbb{R}^m} \mathbb{E}[g(t,x(t),\mathbf{y},\xi)] \label{OCDE-2} &\\
&\quad\quad\quad\quad\textrm{s.t.} \,\,\,\,\mathbf{y}\in K(t,x(t)),
 \end{aligned}
\end{eqnarray}
where $x_0 \in \mathbb{R}^n$ is an initial vector,
and the set-valued function $K: \mathbb{R}_+ \times \mathbb{R}^n \rightrightarrows \mathbb{R}^m$ is defined as below
\[K(t,x(t))\triangleq\{\mathbf{y}\in \mathbb{R}^m: \mathbb{E}[A(\xi)]x(t)+\mathbb{E}[B(\xi)]\mathbf{y}+\mathbb{E}[Q(t,\xi)]\leq 0\}.\]
We assume that all expected values in problem \eqref{OCDE-1}-\eqref{OCDE-2} are well defined.

Let $\|\cdot\|$  denote the Euclidean norm of a vector and a matrix. We suppose that there exists a measurable function $\kappa_f:\Xi\rightarrow\mathbb{R}_+$ with $\mathbb{E}[\kappa_f(\xi)]<\infty$ such that
 for any $t_1$, $t_2\in \mathbb{R}_+$, $u_1$, $u_2\in \mathbb{R}^n$, $v_1$, $v_2\in \mathbb{R}^m$ and almost everywhere (a.e.) $\xi\in\Xi$,
\begin{eqnarray}\label{Lipschitz}
\begin{aligned}
&\,\,\,\,\,\,\|f(t_1,u_1,v_1,\xi)-f(t_2,u_2,v_2,\xi)\|\leq \kappa_f(\xi)(|t_1-t_2|+\|u_1-u_2\|+\|v_1-v_2\|).&\\
\end{aligned}
\end{eqnarray}
We also assume that $g(t,x(t),\cdot,\xi)$ is convex for any $t\in \mathbb{R}_+$, $x(t)\in \mathbb{R}^n$ and a.e. $\xi\in\Xi$, the functions $g(\cdot,\cdot,\cdot,\xi)$ and $Q(\cdot,\xi)$ are both continuous for a.e. $\xi\in\Xi$, and there exists a measurable function $\kappa_Q:\Xi\rightarrow\mathbb{R}_+$ with $\mathbb{E}[\kappa_Q(\xi)]<\infty$ such that
  $\|Q(t,\xi)\|\leq \kappa_Q(\xi)t$ for any $t\in \mathbb{R}_+$ and a.e. $\xi\in\Xi$. Assume that $g(\cdot,\cdot,\cdot,\xi)$ is dominated by an integrable function for a.e. $\xi\in\Xi$.
Then we know $\mathbb{E}[f(t,x(t),y(t),\xi)]$ is Lipschitz continuous and $\mathbb{E}[g(t,x(t),y(t),\xi)]$ is continuous with respect to (w.r.t.) $(t,x(t),y(t))$.

The optimization problem \eqref{OCDE-2} is a stochastic convex program for any fixed $t\in\mathbb{R}_+$ and $x(t)\in \mathbb{R}^n$, since the objective function $\mathbb{E}[g(t,x(t),\cdot,\xi)]$ is convex and the feasible set $K(t,x(t))$ is a convex set. The objective function $\mathbb{E}[g(t,x(t),\cdot,\xi)]$ is not necessarily differentiable and the solution set of \eqref{OCDE-2} may have multiple elements.  Problem \eqref{OCDE-1}-\eqref{OCDE-2} can be equivalently written as the following dynamic generalized stochastic variational inequality:
\begin{eqnarray}\label{ocde-gvi}
\left\{
\begin{aligned}
& \dot{x}(t)=\mathbb{E}[f(t,x(t),y(t),\xi)],\,\,\,\,x(0)=x_0, \\
& 0\in \partial_{y(t)}\mathbb{E}[g(t,x(t),y(t),\xi)] +\mathcal{N}_{K(t,x(t))}(y(t)),
\end{aligned}
\right.
\end{eqnarray}
where $\partial_{y(t)}\mathbb{E}[g(t,x(t),y(t),\xi)]$ is the subdifferential of $\mathbb{E}[g(t,x(t),y(t),\xi)]$ at the point $y(t)$ and $\mathcal{N}_{K(t,x(t))}(y(t))$ denotes the normal cone of $K(t,x(t))$ at $y(t)$ \cite{Rockafellar}.
When the function $\mathbb{E}[g(t,x(t),\cdot,\xi)]$ is continuously differentiable, we can derive a differential stochastic variational inequality (DSVI):
\begin{eqnarray}\label{ocde-vi}
\left\{
\begin{aligned}
& \dot{x}(t)=\mathbb{E}[f(t,x(t),y(t),\xi)],\,\,\,\,x(0)=x_0,  &\\
& 0\in \nabla_{y(t)}\mathbb{E}[g(t,x(t),y(t),\xi)] +\mathcal{N}_{K(t,x(t))}(y(t)).
\end{aligned}
\right.
\end{eqnarray}
It is easy to see that problem  \eqref{ocde-vi} is  a special case of \eqref{ocde-gvi} and problem \eqref{OCDE-1}-\eqref{OCDE-2}. The DSVI (\ref{ocde-vi}) includes the deterministic differential  variational inequality (DVI), which  has many important applications in engineering, economics and biology. The DVI involves dynamics, variational inequalities and equilibrium conditions, and has been studied in \cite{BB2020,chen2012,chen2013,chen2014,pangDVI,chen2018,chen2017,chen2020}.

It should be noted that, when the variational inequality (VI) or optimization problem \eqref{OCDE-2} has multiple solutions, a wrong selection of solutions may make the corresponding ODE unsolvable or numerical scheme divergent.
The authors in \cite{JSpang2009} proposed to use the least-norm solution of the VI to ensure the convergence of time-stepping method for a special class of monotone DVI, which  inspires our regularization approach for problem \eqref{OCDE-1}-\eqref{OCDE-2} in this paper.

The dynamic systems coupled with solutions of an optimization problem have a wide applications in many fields such as atmospheric chemistry \cite{LCH2010,LCH2009} and dynamic flux balance analysis in biological systems \cite{zhaoxiao2017}. They have been extended to stochastic case in \cite{PCM2019}, where the authors investigated a dynamic flux balance analysis model with uncertainty. As mentioned in \cite{LCH2009}, the deterministic dynamic systems coupled with solutions of a convex optimization problem can be seen as an ODE-constrained optimization problem,  which is proposed to estimate parameters for the ODE in \cite{chung-pop-2017}.  In \cite{LCH2009}, the authors proposed a numerical method for the differential equations coupled with a smooth nonconvex optimization problem and  applied the Karush-Kuhn-Tucker conditions to reformulate the problem as the DVI. However, as we mentioned before, when the objective function is nonsmooth, we cannot transform problem \eqref{OCDE-1}-\eqref{OCDE-2} as the DVI and apply the existing methods and results. Therefore, we present the existence of solutions, numerical methods, convergence analysis and applications of \eqref{OCDE-1}-\eqref{OCDE-2} in this paper.

The main contributions of this paper are twofold. (i)
We give sufficient conditions for the existence of a solution $(x,y)$ of problem \eqref{OCDE-1}-\eqref{OCDE-2} on $[0,T]$,  where $x$ is absolutely continuous and $y$ is integrable. In addition, if the objective function in
\eqref{OCDE-2} is strongly convex, then problem \eqref{OCDE-1}-\eqref{OCDE-2} has a solution $(x,y)$ over $[0,\tilde{T}]$ with $x$ being continuously differentiable,  and $y$ being continuous for a positive number $\tilde{T}$ and admitting  a linear growth condition.  (ii)
We propose a regularization method to approximate the objective function in \eqref{OCDE-2} by a strongly convex function and show the unique optimal solution of the regularized optimization problem converges to the least-norm optimal solution of \eqref{OCDE-2} when the regularization parameter goes to zero.
Moreover, we prove the existence of solutions to the discrete regularization problem using the sample average approximation (SAA) and the implicit Euler time-stepping scheme.  We show the solution of the approximation problem constructed by the regularization
approach, SAA and time-stepping method converges to a solution of \eqref{OCDE-1}-\eqref{OCDE-2} with probability 1 (w.p.1) by the repeated limits in the order of the regularization parameter goes to zero, the SAA sample size goes to infinity and the time-stepping step size goes to zero.

The paper is organised as follows: Section \ref{se:Existence} deals with the existence of solutions of problem \eqref{OCDE-1}-\eqref{OCDE-2}. Section \ref{se:regularization} studies the existence of solutions of the regularized problem of \eqref{OCDE-1}-\eqref{OCDE-2} and the convergence to the original problem as the regularization parameter approaches to zero. In Section \ref{se:SAA}, we present the existence of solutions of the SAA of \eqref{OCDE-1}-\eqref{OCDE-2} and the convergence analysis.
In Section \ref{se:time-stepping}, we study the convergence of the time-stepping scheme and show the convergence properties of the discrete method using the SAA and the implicit Euler time-stepping scheme. Section \ref{se:numerical-exam} gives a numerical example to illustrate the theoretical results obtained in this paper. And Section \ref{se:application} shows the application of the estimation of the time-varying parameters in ODE. Some final conclusion remarks are presented in Section \ref{se:conclusions}.

\subsection{Notation}

Denote by $\mathcal{B}(v,r)$  the open ball centered by $v\in \mathbb{R}^n$ with the radius of $r$ in the Euclidean norm.
For sets ${S}_1,S_2\subseteq \mathbb{R}^n$, we denote the distance from $v\in \mathbb{R}^n$ to $S_1$ and the deviation of the set $S_1$ from the set $S_2$ by
$\textrm{dist}(v,S_1)=\inf_{v'\in S_1}\|v-v'\|,$
 and $\mathbb{D}(S_1,S_2)=\sup_{v\in S_1}\textrm{dist}(v,S_2)$, respectively. We also define the Hausdorff distance between the set $S_1$ and the set $S_2$ by $\mathbb{H}(S_1,S_2)=\max\{\mathbb{D}(S_1,S_2),\mathbb{D}(S_2,S_1)\}$. We define ${S}_1+{S}_2=\{z_1+z_2: z_1\in {S}_1, z_2\in {S}_2\}$. For a set ${S}$, $\textrm{int}{S}$ denotes the interior of $S$ and $\tau {S}=\{\tau z: z\in {S}\}$ with a scalar $\tau$.
 Let $C^1([a,b])$ and $C^0([a,b])$ be the spaces of continuously
differentiable vector-valued functions and continuous vector-valued functions on
$[a, b]$, respectively.

\section{Existence of solutions}\label{se:Existence}

In this section, we show the existence of solutions to problem \eqref{OCDE-1}-\eqref{OCDE-2}.
\begin{definition}\cite{DI,yasushi2015}
\begin{enumerate}
  \item [(i)] (lower semicontinuity). A set-valued mapping ${\cal S}: \mathbb{R}^{n_1}\rightrightarrows \mathbb{R}^{m_1}$ is lower semicontinuous at $\bar{z}\in \mathbb{R}^{n_1}$ if for any open set $\mathcal{B}_{\cal S}$ with $\mathcal{B}_{\cal S}\cap {\cal S}(\bar{z})\neq \emptyset$, there exists $\sigma>0$ such that ${\cal S}(z)\cap \mathcal{B}_{\cal S} \neq \emptyset$ for any $z\in \mathcal{B}(\bar{z},\sigma)$.
  \item [(ii)] (upper semicontinuity). A set-valued mapping ${\cal S}: \mathbb{R}^{n_1}\rightrightarrows \mathbb{R}^{m_1}$ is upper semicontinuous at $\bar{z}\in \mathbb{R}^{n_1}$ if for any open set $\mathcal{B}_{\cal S}$ with ${\cal S}(\bar{z})\subseteq \mathcal{B}_{\cal S}$, there exists $\sigma>0$ such that ${\cal S}(z)\subseteq \mathcal{B}_{\cal S}$ for any $z\in \mathcal{B}(\bar{z},\sigma)$.
      \end{enumerate}
\end{definition}

A set-valued mapping ${\cal S}$ is said to be continuous if and only if it is both upper and lower semicontinuous. Obviously, upper (lower) semicontinuity is nothing else than continuity if ${\cal S}$ is single-valued.

Let $\mathcal{S}(t,x(t))$ denote the optimal solution set of \eqref{OCDE-2} for fixed $t\in \mathbb{R}_+$ and $x(t)\in \mathbb{R}^n$. For some $T>0$, if there exists $(x,y)\in C^1([0,T]) \times C^0([0,T])$ fulfilling problem \eqref{OCDE-1}-\eqref{OCDE-2}, we call $(x,y)$ a classic solution of problem \eqref{OCDE-1}-\eqref{OCDE-2} on $[0,T]$.
We call $(x,y)$ a weak solution of problem \eqref{OCDE-1}-\eqref{OCDE-2} over $[0,T]$ if $x$ is absolutely continuous and $y$ is integrable over $[0,T]$ with $y(t)\in \mathcal{S}(t,x(t))$ and
\[x(t)=x_0+\int_{0}^t \mathbb{E}[f(\tau,x(\tau),y(\tau),\xi)]d\tau.\]

 We first show the existence of solutions of problem \eqref{OCDE-1}-\eqref{OCDE-2} under the following assumption.
\begin{assumption}\label{assumption-interior-point} The set $\textrm{int} K(0,x_0)$ is not empty, the function $\mathbb{E}[g(0,x_0,\cdot,\xi)]$ is level-bounded over $K(0,x_0)$ (i.e. all sets $\{\mathbf{y}\in K(0,x_0): \mathbb{E}[g(0,x_0,\mathbf{y},\xi)]\leq \alpha \}$ for $\alpha\in \mathbb{R}$ are bounded) and the function $f(t,\mathbf{x},\cdot,\xi)$ is affine for any $t\in \mathbb{R}$, $\mathbf{x}\in \mathbb{R}^n$ and a.e. $\xi\in\Xi$.
\end{assumption}

\begin{theorem}\label{existence-weak}
Suppose that Assumption \ref{assumption-interior-point} holds. Then there exists $T_0>0$ such that problem \eqref{OCDE-1}-\eqref{OCDE-2} has at least a weak solution $(x^*,y^*)$ on $[0,T]$ for any $T\leq T_0$.
\end{theorem}
\begin{proof}
Let
$K_1(\mathbf{x},\mathbf{q})=\{\mathbf{y}\in \mathbb{R}^m: \mathbb{E}[A(\xi)]\mathbf{x}+\mathbb{E}[B(\xi)]\mathbf{y}+\mathbf{q}\leq 0\}$  for   $\mathbf{x}\in \mathbb{R}^n$ and $\mathbf{q}\in \mathbb{R}^q$. Denote $\mathbf{q}_0=\mathbb{E}[Q(0,\xi)]$. It is easy to see $K_1(x_0,\mathbf{q}_0)=K(0,x_0)$.
Since $\textrm{int} K_1(x_0,\mathbf{q}_0)\neq\emptyset$, there are $\breve{\sigma}>0$ and $\breve{\delta}>0$ such that $\textrm{int}K_1(\mathbf{x},\mathbf{q})\neq\emptyset$ for any $(\mathbf{x},\mathbf{q})\in \mathcal{B}(x_0,\breve{\sigma})\times \mathcal{B}(\mathbf{q}_0,\breve{\delta})$.
By the continuity of $\mathbb{E}[Q(\cdot,\xi)]$, we  conclude that there are $\hat{\sigma}>0$ and $\hat{\delta}>0$ such that $\textrm{int}K(t,x(t))\neq\emptyset$ for any $(t,x(t))\in [0,\hat{\sigma}] \times \mathcal{B}(x_0,\hat{\delta})$ with $x\in C^0([0,\hat{\sigma}])$.

It is easy to verify that $K_1(\tau \mathbf{x}_1+(1-\tau) \mathbf{x}_2, \tau \mathbf{q}_1+(1-\tau) \mathbf{q}_2) \supset \tau K_1(\mathbf{x}_1,\mathbf{q}_1) +(1-\tau)K_1(\mathbf{x}_2,\mathbf{q}_2)$ for $\tau \in (0,1)$, $\mathbf{x}_1$, $\mathbf{x}_2\in \mathcal{B}(x_0,\breve{\sigma})$ and $\mathbf{q}_1$, $\mathbf{q}_2\in \mathcal{B}(\mathbf{q}_0,\breve{\delta})$, which means that $K_1$ is graph-convex. Then following from \cite[Corollary 9.34]{Rockafellar}, $K_1$ is strictly continuous at any point of $\mathcal{B}(x_0,\breve{\sigma})\times \mathcal{B}(\mathbf{q}_0,\breve{\delta})$.

Therefore, by the continuity of $\mathbb{E}[g(\cdot,\cdot,\cdot,\xi)]$ and Assumption \ref{assumption-interior-point}, there are two scalars ${\sigma}$ and ${\delta}$ with $\hat{\sigma}\geq \sigma>0$ and $\hat{\delta}\geq \delta>0$ such that $K$ is continuous over $[0,{\sigma}]\times\mathcal{B}(x_0,{\delta})$ and $\mathbb{E}[g(t,x(t),\cdot,\xi)]$ is level-bounded over $K(t,x(t))$ for any $(t,x(t))\in [0,{\sigma}]\times\mathcal{B}(x_0,{\delta})$ with $x\in C^0([0,{\sigma}])$.
According to \cite[Example 1.11]{Rockafellar},
we know that $\mathcal{S}(t,x(t))$ is nonempty and compact for any $(t,x(t))\in [0,\sigma]\times\mathcal{B}(x_0,{\delta})$ with $x\in C^0([0,{\sigma}])$. It then derives that, by  \cite[Theorem 3.1]{yasushi2015}, $\mathcal{S}$ is convex-valued and upper semicontinuous over $[0,\sigma]\times\mathcal{B}(x_0,\delta)$, which means that there exists $\rho_s>0$ (independent of $(t,x(t))$) such that $\sup\{\|\mathbf{y}\|: \mathbf{y}\in \mathcal{S}(t,x(t))\} \leq \rho_{s}$ over $[0,\sigma]\times\mathcal{B}(x_0,\delta)$.

Since $\mathcal{S}$ is convex-valued and $\mathbb{E}[f(t,x(t),\cdot,\xi)]$ is affine for $t\in\mathbb{R}$ and $x(t)\in \mathbb{R}^n$, the set-valued mapping
 \[f_{\mathcal{S}}(t,x(t))=\{\mathbb{E}[f(t,x(t),\mathbf{y},\xi)]: \mathbf{y}\in \mathcal{S}(t,x(t))\}\]
is convex-valued.
Following from the Lipschitz property of $\mathbb{E}[f(\cdot,\cdot,\cdot,\xi)]$, we know that there exists $\check{\rho}_f>0$ such that
$\|\mathbb{E}[f(t,x(t),y(t),\xi)]\|\leq \check{\rho}_f(1+|t|+\|x(t)\|+\|y(t)\|)$.
 Therefore, there exists $\rho_f>0$ such that the following linear growth condition holds for any $(t,x(t))\in [0,\sigma]\times\mathcal{B}(x_0,\delta)$
\begin{equation}\label{linear-growth}
\sup\{\|\mathbb{E}[f(t,x(t),y(t),\xi)]\|: y(t)\in \mathcal{S}(t,x(t)) \}\leq \rho_f(1+\|x(t)\|).
\end{equation}
It is obvious that $f_{\mathcal{S}}$ is closed by the compactness of $\mathcal{S}$. We then know that $f_{\mathcal{S}}$ is upper semicontinuous over $(t,x(t))\in [0,\sigma]\times\mathcal{B}(x_0,\delta)$ since it is bounded on compact sets \cite[Corollary 1 in Section 1 of Chapter 1]{DIJA}.

According to \cite[Theorem 5.1]{DI} and \cite[Lemma 6.1]{pangDVI}, we know that the following differential inclusion
\begin{equation*}
\left\{
\begin{aligned}
& \dot{x}(t)\in f_\mathcal{S}(t,x(t)), &\\
& x(0)=x_0,
\end{aligned}
\right.
\end{equation*}
has at least one absolutely continuous solution $x^*$. According to \cite[Corollary 1 in Section 14 of Chapter 1]{DIJA} and \cite[Lemma 6.3]{pangDVI}, there exists an integrable function $y^*(t)\in \mathcal{S}(t,x^*(t))$ such that
\[x^*(t)= x_0+\int_{0}^t \mathbb{E}[f(\tau,x^*(\tau),y^*(\tau),\xi)] d\tau.\]
Clearly, there exists $\sigma_{T}>0$ such that $x^*(t)\in\mathcal{B}(x_0,\delta)$ for any $t\in[0,\sigma_{T}]$. By choosing $T_0=\min\{\sigma_{T},\sigma\}$, we can conclude the result.
\end{proof}

If optimization problem \eqref{OCDE-2} has equality constraints with
$K(t,x(t))=\{\mathbf{y}\in \mathbb{R}^m: \mathbb{E}[A(\xi)]x(t)+\mathbb{E}[B(\xi)]\mathbf{y}+\mathbb{E}[Q(t,\xi)]= 0\}$,
we can replace
 ``$\textrm{int} K(0,x_0)$ is not empty'' by
``$K(0,x_0)$ is not empty''  in Assumption \ref{assumption-interior-point} and consider the relaxation set $K(t,x(t),\epsilon)\triangleq\{\mathbf{y}\in\mathbb{R}^m: \|\mathbb{E}[A(\xi)]x(t)+\mathbb{E}[B(\xi)]\mathbf{y}+\mathbb{E}[Q(t,\xi)]\|_\infty\leq\epsilon\}$, where $\epsilon\ge 0$ is a scalar.
Since $K(0,x_0)$ is not empty, we have  $\textrm{int} K(0,x_0,\epsilon)$ with $\epsilon >0$ is not empty.
  Let $\tilde{K}(\epsilon)=K(0,x_0,\epsilon)$. It is easy to see that  $\tilde{K}(\epsilon)$ is graph-convex and the graph $\tilde{K}(\epsilon)$ is polyhedral,  which implies from \cite[Example 9.35]{Rockafellar} that $\tilde{K}(\epsilon)$ is Lipschitz continuous w.r.t. $\epsilon$. Therefore, from the function $\mathbb{E}[g(0,x_0,\cdot,\xi)]$ is level-bounded over $K(0,x_0)$, we have that $\mathbb{E}[g(0,x_0,\cdot,\xi)]$ is level-bounded over $K(0,x_0,\epsilon)$ with any sufficiently small $\epsilon>0$, which means that Assumption \ref{assumption-interior-point} holds to the relaxation optimization problem with int$K(0,x_0,\epsilon)\neq \emptyset$ for any sufficiently small $\epsilon>0$.

It is obvious that $K(t,x(t),\epsilon)$ is also Lipschitz continuous w.r.t. $\epsilon$ with any given $t$ and $x(t)$. It means from \cite[Definition 9.26, Corollary 4.7]{Rockafellar} that $K(t,x(t),\epsilon) \downarrow  K(t,x(t))$ as $\epsilon\downarrow0$. Moreover, from  \cite[Proposition 7.4(f), Exercise 7.8(a)]{Rockafellar}, we know  that $\mathbb{E}[g(t,x(t),\cdot,\xi)]+I_{K(t,x(t),\epsilon)}\rightarrow^{epi} \mathbb{E}[g(t,x(t),\cdot,\xi)]+I_{K(t,x(t))}$ as $\epsilon\downarrow0$, where $I_K$ is the indicator function of set $K$. It then concludes by \cite[Theorem 7.33]{Rockafellar} that $\lim_{\epsilon\downarrow0}\mathbb{D}(\mathcal{S}^\epsilon(t,x(t)),\mathcal{S}(t,x(t)))=0$, where $\mathcal{S}(t,x(t))$ and $\mathcal{S}^\epsilon(t,x(t))$ denote the optimal solution sets of optimization problem \eqref{OCDE-2} with equality constraints and its relaxation optimization problem with  $\epsilon>0$, respectively.
Hence by using this relaxation method, the results of this paper are also applicable without assume that int$K(t,x(t))\neq \emptyset$.


\subsection{Existence in the strong convex case}\label{subse:strong-convex}

In this subsection, we consider a special case of \eqref{OCDE-2} where the objective function is strongly convex.
\begin{assumption}\label{assumption-convexity}
There exists a measurable function $\varrho: \Xi\rightarrow\mathbb{R}_{++}$ with $0<\mathbb{E}[\varrho(\xi)]<\infty$ such that for any $\mathbf{y}_1$, $\mathbf{y}_2\in \mathbb{R}^m$ and $\tau\in (0,1)$,
\[g(t,\mathbf{x},(1-\tau)\mathbf{y}_1+\tau \mathbf{y}_2,\xi)\leq (1-\tau)g(t,\mathbf{x},\mathbf{y}_1,\xi)+\tau g(t,\mathbf{x},\mathbf{y}_2,\xi)-\frac{1}{2}\varrho(\xi) \tau(1-\tau) \|\mathbf{y}_1-\mathbf{y}_2\|^2 \]
holds for any fixed $t\in \mathbb{R}_+$, $\mathbf{x}\in \mathbb{R}^n$ and a.e. $\xi\in\Xi$.
\end{assumption}

Assumption \ref{assumption-convexity} means that $g(t,\mathbf{x},\cdot,\xi)$ is strongly convex for any fixed $t\in \mathbb{R}_+$, $\mathbf{x}\in \mathbb{R}^n$ and a.e. $\xi\in\Xi$ and $\mathbb{E}[g(t,\mathbf{x},\cdot,\xi)]$ is also strongly convex.
Under Assumption \ref{assumption-convexity} we have the following result about the existence of solutions of  problem \eqref{OCDE-1}-\eqref{OCDE-2}.

\begin{theorem}\label{existence-strong-convex}
Suppose that Assumption \ref{assumption-convexity} holds and $\textrm{int} K(0,x_0)\neq\emptyset$.
Then there exists $\tilde{T}>0$ such that problem \eqref{OCDE-1}-\eqref{OCDE-2} has a classic solution $(\tilde{x},\tilde{y})$ on $[0,\tilde{T}]$. In addition, there exists ${\rho}>0$ such that
\begin{equation}\label{linear-growth-x}
\|\tilde{y}(t)\|\leq {\rho}(1+|t|+\|\tilde{x}(t)\|) \quad {\rm for }\quad t\in [0, \tilde{T}].
\end{equation}
\end{theorem}
\begin{proof}

Following the proof of Theorem \ref{existence-weak}, there are $\sigma>0$ and $\delta>0$ such that $K(t,x(t))$ is convex and nonempty for any $(t,x(t))\in[0,\sigma]\times \mathcal{B}(x_0,\delta)$ with $x\in C^0([0,\sigma])$. It then derives the existence of a unique optimal solution $\hat{y}(t,x(t))$ for any $(t,x(t))\in[0,\sigma]\times \mathcal{B}(x_0,\delta)$ by the strong convexity of $\mathbb{E}[g(t,x(t),\cdot,\xi)]$. In addition, we can obtain that the optimal solution set $\mathcal{S}$ of the optimization problem \eqref{OCDE-2} is also upper semicontinuous over $[0,\sigma]\times\mathcal{B}(x_0,\delta)$, which means that $\hat{y}(t,x(t))$ is continuous w.r.t. $t$ and $x(t)$ for any $x\in C^0([0,\sigma])$ since $\mathcal{S}$ is single-valued.

Therefore, applying the Peano existence theorem \cite{ODEpeano}, we find that
\begin{equation}\label{Peano}
\left\{
\begin{aligned}
& \dot{x}(t)=\mathbb{E}[f(t,x(t),\hat{y}(t,x(t)),\xi)],&\\
& x(0)=x_0, &\\
\end{aligned}
\right.
\end{equation}
has a solution $\tilde{x}(t)$, where $\tilde{x}\in C^1([0,\sigma])$. Write $\tilde{y}(t)=\hat{y}(t,\tilde{x}(t))$ and then $\tilde{y}\in C^0([0,\sigma])$. Noting
\[\tilde{x}(t)=x_0+\int_{0}^t \mathbb{E}[f(\tau,\tilde{x}(\tau),\tilde{y}(\tau),\xi)] d\tau.\]
Clearly, there exists $\sigma_{\tilde{T}}>0$ such that $\tilde{x}(t)\in\mathcal{B}(x_0,\delta)$ for $t\in[0,\sigma_{\tilde{T}}]$, which derives that problem \eqref{OCDE-1}-\eqref{OCDE-2} has a classic solution $(\tilde{x},\tilde{y})$ on $[0,\tilde{T}]$ with  $\tilde{T}=\min\{\sigma,\sigma_{\tilde{T}}\}$.

Now we prove \eqref{linear-growth-x}.
  Following from \cite[Example 9.14]{Rockafellar} and the continuity and convexity of $\mathbb{E}[g(t,x(t),\cdot,\xi)]$, $\partial_{y(t)} \mathbb{E}[g(t,x(t),y(t),\xi)]$ is nonempty and compact for any $t\in\mathbb{R}_+$, $x(t)\in \mathbb{R}^n$ and $y(t)\in \mathcal{C}$ with a compact subset $\mathcal{C}$ of $\mathbb{R}^m$. Since $\mathbb{E}[g(t,\tilde{x}(t),\cdot,\xi)]$ is strongly convex, the set-valued mapping $\partial_{y(t)} \mathbb{E}[g(t,\tilde{x}(t),\cdot,\xi)]$ is strongly monotone with constant $\mathbb{E}[\varrho(\xi)]$ over $K(t,\tilde{x}(t))$.
To derive \eqref{linear-growth-x}, it suffices to show the linear growth condition of $\hat{y}(t,\tilde{x}(t))$ w.r.t. $t$ and $\tilde{x}(t)$. It is known that $\hat{y}(t,\tilde{x}(t))$  is the unique optimal solution of the optimization problem \eqref{OCDE-2} if and only if $(\hat{y}(t,\tilde{x}(t)),\hat{z}(t,\tilde{x}(t)))$ with $\hat{z}(t,\tilde{x}(t))\in \partial_{y(t)} \mathbb{E}[g(t,\tilde{x}(t),\hat{y}(t,\tilde{x}(t)),\xi)]$ is the unique solution of the generalized variational inequality: find $(\hat{\mathbf{y}},\hat{\mathbf{z}}) $ with $\hat{\mathbf{z}}\in \partial_{y(t)}\mathbb{E}[g(t,\tilde{x}(t),\hat{\mathbf{y}},\xi)]$ such that
$(\mathbf{y}-\hat{\mathbf{y}})^\top \hat{\mathbf{z}}\geq 0$ for any $\mathbf{y}\in K(t,\tilde{x}(t)).$

Note that $K(t,\tilde{x}(t))$ is a polyhedron for any given $t$ and $\tilde{x}(t)$. Let $\tilde{\mathbf{y}}$ be the least-norm element of $K(t,\tilde{x}(t))$. By Hoffman's error bound for linear systems \cite[Lemma 3.2.3]{Facchinei-pangJS-2003}, we know that there exists $\alpha>0$ (independent of $t$) such that $\|\tilde{\mathbf{y}}\|\leq\alpha(1+|t|+\|\tilde{x}(t)\|)$ for all $t$ and $\tilde{x}(t)$ with $K(t,\tilde{x}(t))\neq\emptyset$.
Let $\tilde{\mathbf{z}}\in \partial_{y(t)} \mathbb{E}[g(t,\tilde{x}(t),\tilde{\mathbf{y}},\xi)]$, we have
\begin{eqnarray*}
0\leq (\tilde{\mathbf{y}}-\hat{y}(t,\tilde{x}(t)))^\top\hat{z}(t,\tilde{x}(t)).
\end{eqnarray*}
By the strong monotonicity of $\partial_{y(t)} \mathbb{E}[g(t,\tilde{x}(t),\cdot,\xi)]$, we have
\begin{eqnarray*}
& &\mathbb{E}[\varrho(\xi)] \|\tilde{\mathbf{y}}-\hat{y}(t,\tilde{x}(t))\|^2\leq (\tilde{\mathbf{y}}-\hat{y}(t,\tilde{x}(t)))^\top(\tilde{\mathbf{z}}-\hat{z}(t,\tilde{x}(t)))\\
& &\leq (\tilde{\mathbf{y}}-\hat{y}(t,\tilde{x}(t)))^\top\tilde{\mathbf{z}}\leq \|\tilde{\mathbf{y}}-\hat{y}(t,\tilde{x}(t))\|\|\tilde{\mathbf{z}}\|,
\end{eqnarray*}
which implies that $\|\tilde{\mathbf{y}}-\hat{y}(t,\tilde{x}(t))\|\leq\mathbb{E}[\varrho(\xi)]^{-1}\|\tilde{\mathbf{z}}\|$. By $\|\tilde{\mathbf{y}}\|\leq\alpha(1+|t|+\|\tilde{x}(t)\|)$ and the boundeness of $\partial_{y(t)} \mathbb{E}[g(t,\tilde{x}(t),\tilde{\mathbf{y}},\xi)]$, there exists ${\rho}>0$ such that
$\|\hat{y}(t,\tilde{x}(t))\|\leq {\rho}(1+|t|+\|\tilde{x}(t)\|)$.
\end{proof}

\begin{remark}
Following the proofs of Theorems \ref{existence-weak} and \ref{existence-strong-convex},
The linear growth condition \eqref{linear-growth-x} in Theorem \ref{existence-strong-convex} plays an important role on the subsequent convergence analysis. The paper \cite{Janin1999} investigated a parameterized convex program with linear constraints and a nonsmooth objective function. By assuming the superquadratic and subquadratic growth conditions for the objective function and the Mangasarian-Fromovitz regularity condition (MFC), the authors showed the upper Lipschitz continuity of the unique optimal solution. We can also derive \eqref{linear-growth-x} by the upper Lipschitz continuity of  the optimal solution of problem \eqref{OCDE-2}  w.r.t. $(t,x) $ at the point $(0,x_0)$. Our conditions (conditions of Theorem \ref {existence-strong-convex}) are easier to verify and weaker than the conditions in \cite{Janin1999}.
\end{remark}

The authors in \cite{pangDVI} also established a linear growth condition for the algebraic variable (the solution of a VI) to ensure the convergence of the implicit Euler method for the DVI. Moreover, in \cite{JSpang2009}, Han et al. derived a linear growth condition for the least-norm solution of a monotone linear complementarity problem and proposed an implicit time-stepping method using the least-norm solutions for differential complementarity systems. Without computing the least-norm solution for a monotone DVI, Chen and Wang \cite{chen2013} proposed a regularized time-stepping method for the DVI and provided the corresponding convergence analysis. These results can be extended to problem  \eqref{OCDE-1}-\eqref{OCDE-2} if $g$ is continuously differentiable and independent of $\xi$. This papers focus on the case that $g$ is nonsmooth and random.

\section{Regularization method}\label{se:regularization}
Since $g(t,\mathbf{x},\cdot,\xi)$ is convex for any $t\in\mathbb{R}_+$, $\mathbf{x}\in\mathbb{R}^n$ and a.e. $\xi\in\Xi$, $\mathbb{E}[g(t,\mathbf{x},\cdot,\xi)]$ is convex for any $t\in\mathbb{R}$ and $\mathbf{x}\in \mathbb{R}^n$.
Therefore, we add a regularization term  $\mu\|\mathbf{y}\|^2$ with $\mu>0$ to the objective function in \eqref{OCDE-2} and get the following regularization optimization problem:
\begin{equation}\label{OCDE-2-regularization}
\begin{aligned}
y^\mu(t)= &\arg\min_{\mathbf{y}\in\mathbb{R}^m}g^\mu(t,x(t),\mathbf{y}) &\\
&\textrm{s.t.} \,\,\,\,\mathbf{y}\in K(t,x(t)),
 \end{aligned}
\end{equation}
where $g^\mu(t,x(t),\mathbf{y})= \mathbb{E}[{g}(t,x(t),\mathbf{y},\xi)]+\mu \|\mathbf{y}\|^2$.

Obviously, under the assumption that $\textrm{int} K(0,x_0)\neq\emptyset$, there are $\sigma>0$ and $\delta>0$ such that the optimization problem \eqref{OCDE-2-regularization} has a unique optimal solution $\hat{y}^\mu(t,x(t))$ over $K(t,x(t))$ for any $(t,x(t))\in[0,\sigma]\times \mathcal{B}(x_0,\delta)$ with $x\in C^0([0,\sigma])$.

\begin{proposition}
Suppose that Assumption \ref{assumption-interior-point} holds.
Let $\hat{y}^\mu(t,{x}(t))$ be the unique optimal solution of problem \eqref{OCDE-2-regularization} with some $\mu>0$,  $t\in \mathbb{R}_+$ and $x(t)\in \mathbb{R}^n$.
 Then it holds that
\begin{equation}\label{linear-growth-mu}
\|\hat{y}^\mu(t,{x}(t))\|\leq \min_{\mathbf{y}\in \mathcal{S}(t,x(t))}\|\mathbf{y}\|,
\end{equation}
where $\mathcal{S}(t,x(t))$ is the optimal solution set of problem \eqref{OCDE-2} with $t\in \mathbb{R}_+$ and $x(t)\in \mathbb{R}^n$.
\end{proposition}
\begin{proof}
Since  $\hat{y}^\mu(t,x(t))$ is the unique optimal solution of the optimization problem \eqref{OCDE-2-regularization}, there exists $\hat{z}^\mu(t,x(t)) \in \partial_{y(t)} \mathbb{E}[g(t,x(t),\hat{y}^\mu(t,x(t)),\xi)]$ such that
\[(\mathbf{y}-\hat{y}^\mu(t,x(t)))^\top (\hat{z}^\mu(t,x(t)) +\mu \hat{y}^\mu(t,x(t)))\geq 0, \quad  \forall \, \mathbf{y}\in K(t,x(t)).\]
Let $\bar{z}(t,x(t)) \in \partial_{y(t)} \mathbb{E}[g(t,x(t),\bar{y}(t,x(t)),\xi)]$, where $\bar{y}(t,x(t))$ is the least-norm element of $\mathcal{S}(t,x(t))$. Then we have
\begin{eqnarray}\label{least-1-monotone}
(\bar{y}(t,x(t))-\hat{y}^\mu(t,x(t)))^\top (\hat{z}^\mu(t,x(t))+\mu \hat{y}^\mu(t,x(t)))\geq 0
\end{eqnarray}
and
\begin{eqnarray}\label{least-monotone}
(\hat{y}^\mu(t,x(t))-\bar{y}(t,x(t)))^\top \bar{z}(t,x(t))\geq0.
\end{eqnarray}
Since $\mathbb{E}[g(t,x(t),\cdot,\xi)]$ is convex, the set-valued mapping $\partial_{y(t)} \mathbb{E}[g(t,x(t),\cdot,\xi)]$ is monotone \cite[Theorem 12.17]{Rockafellar}.
Therefore, from \eqref{least-monotone}, we can obtain
\[(\hat{y}^\mu(t,x(t))-\bar{y}(t,x(t)))^\top \hat{z}^\mu(t,x(t))\geq0.\]
We then get from \eqref{least-1-monotone} that $\mu(\bar{y}(t,x(t))-\hat{y}^\mu(t,x(t)))^\top\hat{y}^\mu(t,x(t))\geq 0$, which implies that $\|\hat{y}^\mu(t,x(t))\|\leq\|\bar{y}(t,x(t))\|$.
\end{proof}

\begin{theorem}\label{existence-regularization}
Suppose that the set $\textrm{int} K(0,x_0)$ is not empty. Then there exists $\hat{T}(\mu)>0$ such that problem \eqref{OCDE-1} with \eqref{OCDE-2-regularization} has a solution $(x^\mu,y^\mu)\in C^1([0,\hat{T}(\mu)]) \times C^0([0,\hat{T}(\mu)])$ for any $\mu>0$. Moreover, there is a positive number $\hat{T}_0$ such that $\hat{T}(\mu)\geq \hat{T}_0$ for any $\mu>0$ if the optimal solution set of problem \eqref{OCDE-2} is not empty.
\end{theorem}
\begin{proof}
Similar with the proof of Theorem \ref{existence-strong-convex}, there are $\sigma>0$ and $\bar{\sigma}(\mu)>0$ such that $\hat{T}(\mu)=\min\{\sigma,\bar{\sigma}(\mu)\}$.

Now we illustrate the existence of $\hat{T}_0$. By \eqref{linear-growth-mu}, there is $\rho_\alpha>0$ such that $\|\hat{y}^\mu(t,x(t))\|\leq\rho_\alpha$ for any $t$, $x(t)$ and $\mu$. Obviously, if $\bar{\sigma}(\mu)\geq\sigma$ for any $\mu>0$, we have $\hat{T}_0=\sigma$. If $\bar{\sigma}(\mu)<\sigma$ for some $\mu>0$, we know that there is $\delta_0\in (0, \delta]$ such that $\|x^\mu(\bar{\sigma}(\mu))-x_0\|=\delta_0$. From
$\|x^\mu(t)-x_0\|\leq\delta$
 for any $t\in[0,\bar{\sigma}(\mu))$, we obtain
\begin{eqnarray*}
\begin{aligned}
\delta_0&=\left\|\int_{0}^{\bar{\sigma}(\mu)}\mathbb{E}[f(\tau,{x}^\mu(\tau),\hat{y}^\mu(\tau,{x}^\mu(\tau)),\xi)] d\tau\right\|\leq\int_{0}^{\bar{\sigma}(\mu)}(\check{\rho}_f\tau+\Theta)d\tau&\\
\end{aligned}
\end{eqnarray*}
which means that $\bar{\sigma}(\mu)\geq \sqrt{\Theta^2+2\check{\rho}_f\delta_0}-\Theta>0$, where
$\Theta=\check{\rho}_f(1+\|x_0\|+\delta+\rho_\alpha)$. Therefore, we conclude the desired result.
\end{proof}

In the convergence analysis of the regularization method as $\mu \downarrow 0$, we use the following notations.
 Let $\mathcal{X}_{T}$ denote the space of $n$-dimensional vector-valued continuous functions over $[0,T]$ equipped with the norm
\[\|u\|_s:=\sup_{t\in[0,T]}\|u(t)\|\]
and $\mathcal{Y}_{T}$ denote the space of $m$-dimensional vector-valued square integrable
functions over $[0, T]$ equipped with the norm
\[\|v\|_{L^2}:=\left(\int_{0}^T\|v(t)\|^2 dt\right)^{\frac{1}{2}}.\]
We define the norm for $(u,v)\in \mathcal{X}_{T} \times \mathcal{Y}_{T}$ by
\[\|(u,v)\|_{\mathcal{X}_{T} \times \mathcal{Y}_{T}}=\|u\|_s+\|v\|_{L^2}.\]
Let $\mathfrak{X}_T$ and $\mathfrak{Y}_T$ denote the space of real-valued continuous functions and real-valued square integrable functions over $[0,T]$, respectively. When $n=1$, we have $\mathfrak{X}_T=\mathcal{X}_T$. Let $\mathcal{Z}_{T}$ denote the space of $m$-dimensional vector-valued continuous functions over $[0,T]$.
Similarly, we define
\begin{equation*}
\begin{aligned}
& \|(u,v)\|_{\mathcal{X}_{T} \times \mathcal{Z}_{T}}=\|u\|_s+\sup_{t\in[0,T]}\|v(t)\|,\,\,\forall\,\,(u,v)\in \mathcal{X}_{T} \times \mathcal{Z}_{T}, &\\
& \|(u,v)\|_{\mathcal{X}_{T} \times \mathfrak{X}_{T}}=\|u\|_s+\sup_{t\in[0,T]}|v(t)|,\,\,\forall\,\,(u,v)\in \mathcal{X}_{T} \times \mathfrak{X}_{T}, &\\
& \|(u,v)\|_{\mathcal{X}_{T} \times \mathfrak{Y}_T}=\|u\|_s+\left(\int_{0}^T v^2(\tau)d\tau\right)^{\frac{1}{2}},\,\,\forall\,\,(u,v)\in \mathcal{X}_{T} \times \mathfrak{Y}_T. &\\
\end{aligned}
\end{equation*}

Denote the optimal value function of optimization problem \eqref{OCDE-2} by $g_{min}(t,x(t))$ with $t\in \mathbb{R}_+$ and $x(t)\in\mathbb{R}^n$. According to \cite[Theorem 3.1]{yasushi2015}, we know that $g_{min}$ is continuous over $[0,\sigma]\times\mathcal{B}(x_0,\delta)$ under Assumption \ref{assumption-interior-point} for some $\sigma$ and $\delta$ in the proof of Theorem \ref{existence-weak}.
 Define
\begin{equation*}
\Phi(x,y)(t)=\left(\begin{array}{c}
               x(t)-x_0-\int_{0}^t \mathbb{E}[f(\tau,x(\tau),y(\tau),\xi)] d\tau\\
               \mathbb{E}[g(t,x(t),y(t),\xi)]-g_{min}(t,x(t))
             \end{array}\right).
\end{equation*}
Let some suitable $T$ with $\sigma\geq T>0$ be fixed. Obviously, we have $\Phi(x,y)\in \mathcal{X}_{T} \times \mathfrak{Y}_T$ for any $(x,y)\in \mathcal{X}_{T} \times \mathcal{Y}_{T}$, and $\Phi(x,y)\in \mathcal{X}_{T} \times \mathfrak{X}_{T}$ for any $(x,y)\in \mathcal{X}_{T} \times \mathcal{Z}_{T}$.
Moreover, we know that $\|\Phi(x,y)\|_{\mathcal{X}_{T} \times \mathfrak{Y}_T}=0$ and $y(t)\in K(t,x(t))$ imply that $(x,y)$ is a weak solution of problem \eqref{OCDE-1}-\eqref{OCDE-2}. And for a continuous function $y\in \mathcal{Z}_T$, $\|\Phi(x,y)\|_{\mathcal{X}_{T} \times \mathfrak{X}_{T}}=0$ and $y(t)\in K(t,x(t))$ imply that $(x,y)$ is a classic solution of problem \eqref{OCDE-1}-\eqref{OCDE-2}.
Similarly, let $g_{min}^\mu(t,x(t))$ denote the optimal value function of the optimization problem \eqref{OCDE-2-regularization} with $t\in \mathbb{R}$, $x(t)\in\mathbb{R}^n$ and $\mu>0$, and define
\begin{equation}\label{phi-mu}
\Phi^\mu(x,y)(t)=\left(\begin{array}{c}
               x(t)-x_0-\int_{0}^t \mathbb{E}[f(\tau,x(\tau),y(\tau),\xi)] d\tau\\
               \mathbb{E}[g(t,x(t),y(t),\xi)]+\mu \|y(t)\|^2-g^\mu_{min}(t,x(t))
             \end{array}\right).
\end{equation}
If $(x^\mu,y^\mu)\in C^1([0,{T}]) \times C^0([0,{T}])$ is a solution of problem \eqref{OCDE-1} with \eqref{OCDE-2-regularization}, we have $\|\Phi^\mu(x^\mu,y^\mu)\|_{\mathcal{X}_{T} \times \mathfrak{X}_{T}}=0$ and then $\|\Phi^\mu(x^\mu,y^\mu)\|_{\mathcal{X}_{T} \times \mathfrak{Y}_T}=0$.

Let $U_1$ and $U_2$ be the spaces taken either $U_1=\mathcal{X}_{T} \times \mathfrak{X}_{T}$ or $U_1=\mathcal{X}_{T} \times \mathfrak{Y}_T$, and $U_2=\mathcal{X}_{T} \times \mathcal{Z}_{T}$ or $U_2=\mathcal{X}_{T} \times \mathcal{Y}_{T}$. A sequence $\{\Psi^k\}_{k=1}^\infty$ is said to be epigraphically convergent to a function $\Psi$, denoted by $\Psi^k \rightarrow^{epi} \Psi$, if
\begin{enumerate}
  \item [(i)] $\lim\inf_{k\rightarrow\infty} \Psi^k(x^k,y^k)\geq \Psi(x,y)$ for any sequence $\{(x^k,y^k)\}_{k=1}^\infty \subseteq U_2$ with $(x^k,y^k)\rightarrow (x,y)$ by the norm $\|\cdot\|_{U_2}$;
  \item [(ii)] $\lim\sup_{k\rightarrow\infty} \Psi^k(x^k,y^k)\leq \Psi(x,y)$ for some sequence $\{(x^k,y^k)\}_{k=1}^\infty \subseteq U_2$ with $(x^k,y^k)\rightarrow (x,y)$ by the norm $\|\cdot\|_{U_2}$.
\end{enumerate}
To study the convergence of $\{(x^\mu,y^\mu)\}$ in $U_2$, we firstly have the following lemma about the mapping $\|\Phi^\mu\|_{U_1}$ is epigraphically convergent to $\|\Phi\|_{U_1}$ as $\mu\downarrow 0$.

\begin{lemma}\label{epi-lemma}
Suppose that Assumption \ref{assumption-interior-point} holds.
Let $\{\mu_k\}_{k=1}^\infty \downarrow 0$ be given and $\Phi^k=\Phi^{\mu_k}$ be defined in \eqref{phi-mu}. Then for any sequence $\{(x^{k},y^{k})\}_{k=1}^\infty\subset U_2$ with $(x^{k},y^{k}) \rightarrow (x,y)$ by the norm $\|\cdot\|_{U_2}$ as $k\rightarrow\infty$, we have $\|\Phi^{k}(x^k,y^k)\|_{U_1}\rightarrow\|\Phi(x,y)\|_{U_1}$ and $\|\Phi^{k}\|_{U_1}\rightarrow^{epi}\|\Phi\|_{U_1}$.
\end{lemma}
\begin{proof}
For any given $\mu>0$, $t\in\mathbb{R}_+$ and $x(t)\in\mathbb{R}^n$, it is clear that $g_{min}(t,x(t))\leq g^\mu_{min}(t,x(t))$ as $\hat{y}^\mu(t,x(t))\in K(t,x(t))$. In addition,
\begin{eqnarray*}
\begin{aligned}
g^\mu_{min}(t,x(t))&=\mathbb{E}[g(t,x(t),\hat{y}^\mu(t,x(t)),\xi)]+\mu \|\hat{y}^\mu(t,x(t))\|^2&\\
&= \min_{\mathbf{y}\in K(t,x(t))} \{\mathbb{E}[g(t,x(t),\mathbf{y},\xi)]+\mu \|\mathbf{y}\|^2 \}&\\
&\leq \min_{\mathbf{y}\in \mathcal{S}(t,x(t))}\{ \mathbb{E}[g(t,x(t),\mathbf{y},\xi)]+\mu \|\mathbf{y}\|^2\} &\\
&\leq g_{min}(t,x(t))+  \mu \min_{\mathbf{y}\in \mathcal{S}(t,x(t))}\|\mathbf{y}\|^2,
\end{aligned}
\end{eqnarray*}
which means that
$|g^\mu_{min}(t,x(t))-g_{min}(t,x(t))|\leq \mu \min_{\mathbf{y}\in \mathcal{S}(t,x(t))}\|\mathbf{y}\|^2$
for any given $(t,x(t))\in [0,\sigma]\times\mathcal{B}(x_0,\delta)$, since $\mathbb{E}[g(t,x(t),\hat{y}^\mu(t,x(t)),\xi)]\geq g_{min}(t,x(t)) $ and \\ $\hat{y}^\mu(t,x(t))$
$\in K(t,x(t))$.
By the uniform boundedness of $\mathcal{S}(t,x(t))$ for any $(t,x(t))\in [0,\sigma]\times\mathcal{B}(x_0,\delta)$ and $g^{\mu_1}_{min}(t,x(t))\leq g^{\mu_2}_{min}(t,x(t))$ for $\mu_1\le \mu_2$,   we can obtain that $g^\mu_{min}$ converges to $g_{min}$ uniformly as $\mu\downarrow0$ over $(t,x(t))\in [0,\sigma]\times\mathcal{B}(x_0,\delta)$.

Let $(x^{k},y^{k})\rightarrow (x,y)$ by the norm $\|\cdot\|_{U_2}$ as $k\rightarrow\infty$.
Taking $U_1=\mathcal{X}_{T} \times \mathfrak{X}_{T}$ and $U_2=\mathcal{X}_{T} \times \mathcal{Z}_{T}$, we have
\begin{eqnarray*}
\begin{aligned}
 &\|\Phi^{k}(x^{k},y^{k})-\Phi(x^{k},y^{k})\|_{U_1} \\
 &\leq\mu_k \sup_{t\in[0,T]}\|y^{k}(t,x^{k}(t))\|^2+ \sup_{t\in[0,T]}|g^{\mu_k}_{min}(t,x^{k}(t))-g_{min}(t,x^{k}(t))|\to 0  \quad  {\rm as}\quad   \mu_k\downarrow 0.
\end{aligned}
\end{eqnarray*}
If we take $U_1=\mathcal{X}_{T}\times \mathfrak{Y}_T$ and $U_2=\mathcal{X}_{T} \times \mathcal{Y}_{T}$, we have
\begin{eqnarray*}
\begin{aligned}
 &\|\Phi^{k}(x^{k},y^{k})-\Phi(x^{k},y^{k})\|_{U_1} \\
 &\leq\mu_k \|y^{k}(\cdot,x^{k})\|^2_{L^2}+ \left(\int_{0}^T(g^{\mu_k}_{min}(t,x^{k}(t))-g_{min}(t,x^{k}(t)))^2dt\right)^{\frac{1}{2}}\to 0  \quad  {\rm as}\quad   \mu_k\downarrow 0.
\end{aligned}
\end{eqnarray*}
Moreover $\|\Phi(x^{k},y^{k})\|_{U_1}\rightarrow\|\Phi(x,y)\|_{U_1}$ as $k\rightarrow\infty$ since $\|\Phi\|_{U_1}$ is continuous. We then obtain
$\|\Phi^{k} (x^{k},y^{k})\|_{U_1}\rightarrow\|\Phi(x,y)\|_{U_1}$ by
\begin{eqnarray*}
\begin{aligned}
&\|\Phi^{k}(x^{k},y^{k})-\Phi(x,y)\|_{U_1}\\
&\leq \|\Phi^{k}(x^{k},y^{k})-\Phi(x^{k},y^{k})\|_{U_1}+\|\Phi(x^{k},y^{k})-\Phi(x,y)\|_{U_1}.
\end{aligned}
\end{eqnarray*}
It then implies that $\|\Phi^{k}\|_{U_1}\rightarrow^{epi}\|\Phi\|_{U_1}$.
\end{proof}

\begin{theorem}\label{convergence-regularization}
Suppose that Assumption \ref{assumption-interior-point} holds. Let $(x^\mu,y^\mu)\in C^1([0,\hat{T}_{0}])\times C^0([0,\hat{T}_{0}])$ be a solution of problem \eqref{OCDE-1} with \eqref{OCDE-2-regularization} for any $\mu>0$. Then there exists a sequence $\{\mu_k\}_{k=1}^\infty\downarrow0$ such that $x^{\mu_k}\rightarrow x^*$ as $k\rightarrow\infty$ uniformly over $[0,\hat{T}_{0}]$ and $y^{\mu_k}\rightarrow y^*$ as $k\rightarrow\infty$ weakly in $\mathcal{Y}_{\hat{T}_{0}}$. In addition,
\begin{enumerate}
  \item [(i)] if $y^{\mu_k}\rightarrow y^*$ w.r.t. $\|\cdot\|_{L^2}$ as $k\rightarrow\infty$, then $(x^*,y^*)$ is a weak solution of \eqref{OCDE-1}-\eqref{OCDE-2} over $[0,\hat{T}_{0}]$;
  \item [(ii)] if $y^{\mu_k}\rightarrow y^*$ uniformly as $k\rightarrow\infty$, then $(x^*,y^*)$ is a classic solution of \eqref{OCDE-1}-\eqref{OCDE-2} over $[0,\hat{T}_{0}]$; moreover, $y^*(t)$ is the unique least-norm optimal solution of problem \eqref{OCDE-2} with $t$ and $x^*(t)$.
\end{enumerate}
\end{theorem}
\begin{proof}
Notice that the Lipschitz property (\ref{Lipschitz}) implies that $\mathbb{E}[f(\cdot,\cdot,\cdot,\xi)]$ has linear growth in $(t,\mathbf{x},\mathbf{y})\in\mathbb{R}\times\mathbb{R}^n \times \mathbb{R}^m$, i.e., there exists $\check{\rho}_f>0$ such that $\|\mathbb{E}[f(t,\mathbf{x},\mathbf{y},\xi)]\|\leq\check{\rho}_f(1+|t|+\|\mathbf{x}\|+\|\mathbf{y}\|)$.
Let $(x^\mu,y^\mu)\in C^1([0,\hat{T}_{0}])\times C^0([0,\hat{T}_{0}])$ be a solution of problem \eqref{OCDE-1} with \eqref{OCDE-2-regularization} for any $\mu>0$. According to \eqref{linear-growth-mu} and the compactness of the optimal solution set of problem \eqref{OCDE-2}, there is $\rho_\alpha>0$ (independent of $\mu$, $t$ and $x(t)$) such that $\|y^\mu(t)\|\leq {\rho}_\alpha$. We then have
\begin{eqnarray*}
\begin{aligned}
\|x^\mu(t)\|&\leq \|x_0\|+\int_{0}^t \|\mathbb{E}[f(\tau,x^\mu(\tau),y^\mu(\tau),\xi)]\| d\tau&\\
&\leq \|x_0\|+\check{\rho}_f\int_{0}^t(1+\rho_\alpha+|\tau|+\|x^\mu(\tau)\|)d\tau,
\end{aligned}
\end{eqnarray*}
which implies that for any $t\in[0,\hat{T}_{0}]$ there exists $\bar{\rho}_f>0$ such that
\[\|x^\mu(t)\|\leq \|x_0\|+\bar{\rho}_f\int_{0}^t(1+\|x^\mu(\tau)\|)d\tau.\]
We then have $\|x^\mu\|_s\leq (1+\|x_0\|) \exp(\bar{\rho}_f\hat{T}_{0})-1$, according to \cite[Lemma 2.6]{chen2013}.
Hence, $\{x^\mu\}$ is uniformly bounded on $[0,\hat{T}_{0}]$ for any $\mu>0$ and then so is $\{\dot{x}^\mu\}$, which means that $\{x^\mu\}$ is equicontinuous over $[0,\hat{T}_{0}]$ for any $\mu>0$. By Arzel\'{a}-Ascoli theorem \cite{AA15}, there exists a sequence $\{\mu_k\}_{k=1}^\infty\downarrow0$ such that $\{x^{\mu_k}\}$ is convergent to a point $x^*\in \mathcal{X}_{\hat{T}_{0}}$ uniformly over $[0,\hat{T}_{0}]$.

In addition, we know that $\{y^\mu\}$ is uniformly bounded on $[0,\hat{T}_{0}]$ for any $\mu>0$ by $\|y^\mu(t)\|\leq {\rho}_\alpha(1+|t|+\|x^\mu(t)\|)$. By Alaglu's theorem \cite{AA15}, there exists a subsequence of $\{y^{\mu_k}\}$, which we may assume without loss of generality to
be $\{y^{\mu_k}\}$ itself, has a weak* limit, named $y^*$, in $\mathcal{Y}_{\hat{T}_{0}}$. Since $\mathcal{Y}_{\hat{T}_{0}}$ is a Hilbert space, it is a reflexive Banach space, which implies that weak* convergent sequences are also weakly convergent sequences.

By \cite[Example 9.35]{Rockafellar}, we know that $K(t,\cdot)$ is Lipschitz continuous on its domain for any $t\in \mathbb{R}$, which means that $\mathbb{H}({K}(t,x^\mu(t)),{K}(t,x^*(t)))\rightarrow0$ as $x^\mu\rightarrow x^*$ uniformly.
Then by $y^\mu(t)\in {K}(t,x^\mu(t))$, we have $y^*(t)\in {K}(t,x^*(t))$.
Then following Lemma \ref{epi-lemma}, we know that if $y^{\mu_k}\rightarrow y^*$ w.r.t. $\|\cdot\|_{L^2}$, $(x^*,y^*)$ is a weak solution of \eqref{OCDE-1}-\eqref{OCDE-2} over $[0,\hat{T}_{0}]$; if $y^{\mu_k}\rightarrow y^*$ uniformly as $k\rightarrow\infty$, then $(x^*,y^*)$ is a classic solution of \eqref{OCDE-1}-\eqref{OCDE-2} over $[0,\hat{T}_{0}]$.

Let $\hat{y}^{\mu_k}(t,{x}(t))$ denote the unique optimal solution of problem \eqref{OCDE-2-regularization} with any $\mu_k>0$, $t\in \mathbb{R}_+$ and $x(t)\in \mathbb{R}^n$. Then by $y^{\mu_k}(t)=\hat{y}^{\mu_k}(t,{x}^{\mu_k}(t))$, $y^{\mu_k}\rightarrow y^*$ uniformly as $k\rightarrow\infty$, and the continuity of $\hat{y}^{\mu_k}$, we know that $\lim_{k\rightarrow\infty}\|\hat{y}^{\mu_k}(t,x^*(t))-y^*(t)\|=0$.
Since $(x^*,y^*)$ is a classic solution of \eqref{OCDE-1}-\eqref{OCDE-2}, we obtain that $y^*(t)\in \mathcal{S}(t,x^*(t))$ and then $y^*(t)$ is the unique least-norm element of $\mathcal{S}(t,x^*(t))$ by \eqref{linear-growth-mu}.
\end{proof}

When  there exists a constant $\tilde{\varrho}>0$ such that for any $\mathbf{y}_1$, $\mathbf{y}_2\in \mathbb{R}^m$ and $\tau\in (0,1)$,
\begin{eqnarray*}
\begin{aligned}
\mathbb{E}[g(t,\mathbf{x},(1-\tau)\mathbf{y}_1+\tau \mathbf{y}_2,\xi)]\leq& (1-\tau)\mathbb{E}[g(t,\mathbf{x},\mathbf{y}_1,\xi)]+\tau \mathbb{E}[g(t,\mathbf{x},\mathbf{y}_2,\xi)]&\\
&-\frac{1}{2}\tilde{\varrho} \tau(1-\tau) \|\mathbf{y}_1-\mathbf{y}_2\|^2
\end{aligned}
\end{eqnarray*}
holds for any fixed $t\in \mathbb{R}_+$ and $\mathbf{x}\in \mathbb{R}^n$, the objective function of optimization problem \eqref{OCDE-2} is strongly convex w.r.t. $\mathbf{y}$, and \eqref{OCDE-2} admits a unique optimal solution $y^*$ which is continuous w.r.t. $t$ for any $x\in C^0([0,\sigma])$ by Theorem \ref{existence-strong-convex}. Fix some $x\in C^0([0,\sigma])$ and let $z^\mu(t)\in\partial_{y(t)} \mathbb{E}[g(t,x(t),y^\mu(t),\xi)]$ and $z^*(t)\in\partial_{y(t)} \mathbb{E}[g(t,x(t),y^*(t),\xi)]$.
Then we have
\begin{eqnarray}\label{least-1-monotone-mu}
(y^*(t)-{y}^\mu(t))^\top ({z}^\mu(t)+\mu {y}^\mu(t))\geq 0\textrm{ and }({y}^\mu(t)-{y}^*(t))^\top {z}^*(t)\geq0.
\end{eqnarray}
In this case, by the strong monotonicity of $\partial_{y(t)} \mathbb{E}[g(t,x(t),\cdot,\xi)]$ and \eqref{least-1-monotone-mu} and \eqref{linear-growth-mu}, we have
\begin{equation*}
\begin{aligned}
\tilde{\varrho}\|y^\mu(t)-y^*(t)\|^2&\leq (y^\mu(t)-y^*(t))^\top (z^\mu(t)-z^*(t))\leq (y^\mu(t)-y^*(t))^\top z^\mu(t)&\\
&\leq\mu (y^*(t)-y^\mu(t))^\top y^\mu(t)\leq \mu \|y^*(t)-y^\mu(t)\|\|y^\mu(t)\|&\\
&\leq\mu \|y^*(t)-y^\mu(t)\|\|y^*(t)\|.
\end{aligned}
\end{equation*}

 We then obtain that the pointwise convergence of $y^\mu$ to $y^*$ as $\mu\downarrow0$, which means that $y^\mu\rightarrow y^*$ w.r.t. $\|\cdot\|_{L^2}$ as $\mu\downarrow0$ and the uniform convergence by the continuity of $y^\mu$ and $y^*$.


\section{Sample average approximation}\label{se:SAA}

We apply the sample average approximation (SAA) approach to solve problem \eqref{OCDE-1}-\eqref{OCDE-2}. We consider an independent identically distributed (i.i.d) sample of $\xi(\omega)$, which is
denoted by $\{\xi_1,\cdot\cdot\cdot,\xi_\nu\}$, and use the following SAA problem to
approximate problem \eqref{OCDE-1}-\eqref{OCDE-2}:
\begin{eqnarray}
& &\dot{x}(t)=\frac{1}{\nu}\sum_{\ell=1}^\nu f(t,x(t),y(t),\xi_\ell),\,\,\, x(0)=x_0, \label{OCDE-1-SAA}\\
& &\begin{aligned}
&y(t)\in\arg\min_{\mathbf{y}\in\mathbb{R}^m} \frac{1}{\nu}\sum_{\ell=1}^\nu {g}(t,x(t),\mathbf{y},\xi_\ell) \label{OCDE-2-SAA}&\\
& \quad\quad\quad\quad\textrm{s.t.} \,\,\,\, \mathbf{y}\in K^\nu(t,x(t)),\\
 \end{aligned}
\end{eqnarray}
where \[K^\nu(t,x(t))\triangleq\left\{\mathbf{y}\in \mathbb{R}^m: \frac{1}{\nu}\sum_{\ell=1}^\nu (A(\xi_\ell)x(t)+B(\xi_\ell)\mathbf{y}+Q(t,\xi_\ell))\leq 0\right\}.\]

In this paper, by saying a property holds with probability 1 (w.p.1) for sufficiently large $\nu$, we mean that there exists a set $\Omega_0\subset \Omega$ of $\mathcal{P}$-measure zero such that for all $\omega\in \Omega\backslash\Omega_0$ there exists a positive integer $\nu^*(\omega)$ such that the property holds for all $\nu\geq\nu^*(\omega)$.

\begin{theorem}\label{existence-saa}
Suppose that Assumption \ref{assumption-convexity} holds and $\textrm{int} K(0,x_0)\neq\emptyset$.
 Then there exists $T^*>0$ such that problem \eqref{OCDE-1-SAA}-\eqref{OCDE-2-SAA} has a solution $(x^\nu,y^\nu)\in C^1([0,T^*])\times C^0([0,T^*])$ w.p.1 for sufficiently large $\nu$. Moreover, there exists ${\rho}^*>0$ such that
\begin{equation}\label{linear-growth-nu}
\|{y}^\nu(t)\|\leq {\rho}^*(1+|t|+\|{x}^\nu(t)\|), \quad {\rm for }\quad t\in[0,T^*]
\end{equation}
w.p.1 for sufficiently large $\nu$.
\end{theorem}
\begin{proof}

By $\textrm{int} K(0,x_0)\neq\emptyset$ and the strong Law of Large Number, we can conclude $\textrm{int} K^\nu(0,x_0)\neq\emptyset$ w.p.1 for sufficiently large $\nu$.

Similar with the proof of Theorem \ref{existence-weak}, we can also conclude that there are $\sigma_1>0$ and $\delta_1>0$ such that $\textrm{int} K^\nu(t,x(t))\neq\emptyset$ for any $(t,x(t))\in [0,\sigma_1]\times \mathcal{B}(x_0,\delta_1)$ w.p.1 for sufficiently large $\nu$.

 Assumption \ref{assumption-convexity} implies that  $\frac{1}{\nu}\sum_{\ell=1}^\nu g(t,x(t),\cdot,\xi_\ell)$ is  strongly convex w.p.1 for sufficiently large $\nu$.
 Similar with the proof of Theorem \ref{existence-strong-convex}, we obtain our results.
\end{proof}

\begin{theorem}\label{convergence-saa}
Suppose that Assumption \ref{assumption-convexity} holds and $\textrm{int} K(0,x_0)\neq\emptyset$.
Let $(x^\nu,y^\nu)\in C^1([0,T^*])\times C^0([0,T^*])$ be a solution of problem \eqref{OCDE-1-SAA}-\eqref{OCDE-2-SAA}. Then there are $\bar{T}$ with $T^*\geq \bar{T} >0$ and a sequence $\{\nu_k\}_{k=1}^\infty$ with $\nu_k \rightarrow\infty$ such that $x^{\nu_k}\rightarrow x^*$ as $k\rightarrow\infty$  w.p.1 uniformly over $[0,\bar{T}]$ and $y^{\nu_k}\rightarrow y^*$ w.p.1 as $k\rightarrow\infty$ w.r.t. $\|\cdot\|_{L^2}$ in $\mathcal{Y}_{\bar{T}}$, where $(x^*,y^*)$ is a weak solution of \eqref{OCDE-1}-\eqref{OCDE-2} over $[0,\bar{T}]$.
\end{theorem}
\begin{proof}
Since $(x^\nu,y^\nu)\in C^1([0,T^*])\times C^0([0,T^*])$ is a solution of problem \eqref{OCDE-1-SAA}-\eqref{OCDE-2-SAA}, by the linear growth condition \eqref{linear-growth-nu}, we obtain that there exists $\hat{\rho}_f>0$ such that
\[\|x^\nu(t)\|\leq \|x_0\|+\hat{\rho}_f\int_{0}^t(1+\|x^\nu(\tau)\|)d\tau\]
holds w.p.1 for sufficiently large $\nu$, which implies that $\|x^\nu\|_s\leq (1+\|x_0\|) \exp(\hat {\rho}_fT^*)-1$
w.p.1 for sufficiently large $\nu$.
Hence, we obtain that $\{x^\nu\}$ is uniformly bounded w.p.1 for sufficiently large $\nu$ and then so is $\{\dot{x}^\nu\}$, which means that $\{x^\nu\}$ is equicontinuous over $[0,T^*]$ w.p.1 for sufficiently large $\nu$. By Arzel\'{a}-Ascoli theorem, there exists a sequence $\{\nu_k\}$ such that $\{x^{\nu_k}\}$ is convergent to a point $x^*\in \mathcal{X}_{T^*}$  as $\nu_k\rightarrow\infty$ w.p.1 uniformly over $[0,T^*]$.

 Let $\hat{y}(t,x(t))$ and $\hat{y}^\nu(t,x(t))$ denote the optimal solutions of optimization problems \eqref{OCDE-2} and \eqref{OCDE-2-SAA} with $t$ and $x(t)$, respectively. Theorems \ref{existence-strong-convex} and \ref{existence-saa} imply that there are $\tilde{\sigma}>0$ and $\tilde{\delta}>0$ such that $\hat{y}$ and $\hat{y}^\nu$ are bounded and $\textrm{int} K(t,x(t))\neq\emptyset$ over $(t,x(t))\in[0,\tilde{\sigma}]\times \mathcal{B}(x_0,\tilde{\delta})$.
 It implies that there exists a compact set $\mathcal{C}\subseteq\mathbb{R}^m$ such that $\hat{y}(t,x(t))\in \mathcal{C}$ and $\hat{y}^\nu(t,x(t))\in \mathcal{C}$ w.p.1 for sufficiently large $\nu$.
 Following from \cite[Theorem 7.53]{Shapiro2014}, we obtain that for any given $t\in\mathbb{R}_+$ and $x(t)\in \mathbb{R}^n$, $\frac{1}{\nu}\sum_{\ell=1}^\nu g(t,x(t),\mathbf{y},\xi_\ell)$ converges to $\mathbb{E}[g(t,x(t),\mathbf{y},\xi)]$ w.p.1 uniformly on $\mathbf{y}\in\mathcal{C}$ as $\nu\rightarrow\infty$.
 In addition, following from the strong Large Law of Number, we can obtain that $\mathbb{D}(K^\nu(t,x(t)), K(t,x(t)))\rightarrow 0$ w.p.1 as $\nu\rightarrow\infty$ for any $t\in\mathbb{R}_+$ and $x(t)\in \mathbb{R}^n$, which means that for any fixed $t$ and $x(t)$ if $\mathbf{y}^\nu\in K^\nu(t,x(t))$ and $\mathbf{y}^\nu$ converges w.p.1 to a point $\mathbf{y}$, then $\mathbf{y}\in K(t,x(t))$. According to \cite[Remark 8]{Shapiro2014}, we conclude that there exists a sequence $\{y^\nu(t,x(t))\}$ with $y^\nu(t,x(t))\in  K^\nu(t,x(t))$ such that $y^\nu(t,x(t))\rightarrow \hat{y}(t,x(t))$ w.p.1 as $\nu\rightarrow\infty$ since $\textrm{int} K(t,x(t))\neq\emptyset$ for any $(t,x(t))\in [0,\tilde{\sigma}]\times\mathcal{B}(x_0,\tilde{\delta})$.  Therefore, following from \cite[Theorem 5.5]{Shapiro2014}, we can obtain that $\hat{y}^\nu(t,x(t))$ converges to $\hat{y}(t,x(t))$ w.p.1 as $\nu\rightarrow\infty$ on $(t,x(t))\in[0,\tilde{\sigma}]\times \mathcal{B}(x_0,\tilde{\delta})$.
 According to Lebesgue Dominated Convergence Theorem and \eqref{linear-growth-nu},
we then conclude that there exists $\tilde{T}>0$ such that $\hat{y}^\nu(\cdot,x)$ converges to $\hat{y}(\cdot,x)$ w.p.1 w.r.t. $\|\cdot\|_{L^2}$ as $\nu\rightarrow\infty$, where $x$ is continuously differentiable over $t\in[0,\tilde{T}]$.

Let $\bar{T}=\min\{T^*,\tilde{T}\}$.
It is clear that $y^\nu(t)=\hat{y}^\nu(t,x^\nu(t))$ and $\hat{y}^{\nu_k}(\cdot,x^{\nu_k})$ converges to $\hat{y}(\cdot,x^*)$ w.p.1 w.r.t. $\|\cdot\|_{L^2}$ as $k\rightarrow\infty$, following from the continuity of $\hat{y}^\nu$ and $\hat{y}$. Denote $y^*(t)=\hat{y}(t,x^*(t))$. Then  taking $\{\nu_k\}$ with $\nu_k\rightarrow\infty$, for any $t\in[0,\bar{T}]$, we have
\begin{eqnarray*}
\begin{aligned}
& \left\|\int_{0}^t \frac{1}{{\nu_k}}\sum_{\ell=1}^{\nu_k} f(\tau,x^{\nu_k}(\tau),y^{\nu_k}(\tau),\xi_\ell) d\tau - \int_{0}^t \mathbb{E}[f(\tau,x^*(\tau),y^*(\tau),\xi)]d\tau\right\| &\\
&  \leq  \frac{1}{\nu_k}\sum_{\ell=1}^{\nu_k}\kappa_f(\xi_\ell)\left(\bar{T}\|x^{\nu_k}-x^*\|_s +\sqrt{\bar{T}}\|y^{\nu_k}-y^*\|_{L^2}\right)+\mathcal{L}_{\bar{T}},
\end{aligned}
\end{eqnarray*}
where
 \[\mathcal{L}_{\bar{T}}=\bar{T}\left\| \frac{1}{{\nu_k}}\sum_{\ell=1}^{\nu_k} f(\cdot,x^*,y^*,\xi_\ell) - \mathbb{E}[f(\cdot,x^*,y^*,\xi)]\right\|_s.\]
Similarly, by \cite[Theorem 7.53]{Shapiro2014}, we obtain that $\frac{1}{\nu}\sum_{\ell=1}^\nu f(t,x^*(t),y^*(t),\xi_\ell)$ converges to $\mathbb{E}[f(t,x^*(t),y^*(t),\xi)]$ w.p.1 uniformly on $t\in[0,\bar{T}]$ as $\nu\rightarrow\infty$.
Therefore, we can conclude that
\[x^*(t)=x_0+\int_{0}^t \mathbb{E}[f(\tau,x^*(\tau),y^*(\tau),\xi)]d\tau\]
 by $x^{\nu_k}(t)=x_0+\int_{0}^t \frac{1}{\nu_k}\sum_{\ell=1}^{\nu_k} f(\tau,x^{\nu_k}(\tau),y^{\nu_k}(\tau),\xi)d\tau.$ By $x^*\in \mathcal{X}_{\bar{T}}$, we obtain $y^*\in \mathcal{Y}_{\bar{T}}$, which means that $(x^*,y^*)$ is a weak solution of problem \eqref{OCDE-1}-\eqref{OCDE-2} over $[0,\bar{T}]$.
\end{proof}

For the case that $\mathbb{E}[g(t,\mathbf{x},\cdot,\xi)]$ is convex, we can choose a measurable function $\hat{\varrho}: \Xi\rightarrow \mathbb{R}_{++}$ with $0< \mathbb{E}[\tilde{\rho}(\xi)]<\infty$ and consider the regularized function
\[\tilde{g}(t,x(t),y(t),\xi)=g(t,x(t),y(t),\xi)+\frac{\mu}{2}\tilde{\rho}(\xi)\|y(t)\|^2.\]
Then Assumption \ref{assumption-convexity} holds for $\tilde{g}$ with $\mu\tilde{\rho}$ and $\mu>0$.
We apply SAA method to \eqref{OCDE-1}-\eqref{OCDE-2} with $\tilde{g}$  and obtain
\begin{equation}\label{OCDE-2-SAA-regularization}
y^\nu_\mu(t)=\arg\min_{\mathbf{y}\in K^\nu(t,x(t))}\frac{1}{\nu}\sum_{\ell=1}^\nu {g}(t,x(t),\mathbf{y},\xi_\ell)+\frac{\mu}{2\nu} \sum^\nu_{\ell=1}\tilde{\rho}(\xi_\ell) \|\mathbf{y} \|^2.
\end{equation}
According to Theorems \ref{convergence-regularization} and \ref{convergence-saa}, we can obtain the following result.
\begin{theorem}\label{convergence-saa-regularization}
Suppose that Assumption \ref{assumption-interior-point} holds.
Let $(x_\mu^\nu,y_\mu^\nu)$ be a solution of problem \eqref{OCDE-1-SAA} with \eqref{OCDE-2-SAA-regularization} for some $\mu>0$ and $\nu>0$. Then there are $\check{T}>0$, $(x^*, y^*)\in C^1([0,\check{T}])\times C^0([0,\check{T}])$, a sequence $\{\mu_k\}_{k=1}^\infty$ with $\mu_k \downarrow0$ and a sequence $\{\nu_k\}_{k=1}^\infty$ with $\nu_k \rightarrow\infty$
 such that
  \[\lim_{\mu_k\downarrow0}\lim_{\nu_k\rightarrow\infty}\|x_{\mu_k}^{\nu_k}-x^*\|_s=0,\,\,\,w.p.1\]
and $y^{\nu_k}_{\mu_k}\rightarrow y^*$ weakly w.p.1 in $\mathcal{Y}_{\check{T}}$ by the order of $\mu_k \downarrow0$ and $\nu_k \rightarrow\infty$. If
   \[\lim_{\mu_k\downarrow0}\lim_{\nu_k\rightarrow\infty}\|y_{\mu_k}^{\nu_k}-y^*\|_{L^2}=0,\,\,\,w.p.1\]
   then $(x^*,y^*)$ is a weak solution of \eqref{OCDE-1}-\eqref{OCDE-2} over $[0,\check{T}]$.
\end{theorem}

\section{Time-stepping method}\label{se:time-stepping}

We now adopt the time-stepping method for solving problem \eqref{OCDE-1-SAA}-\eqref{OCDE-2-SAA} with a fixed sample $\{\xi_1,\ldots, \xi_\nu\}$, which uses a finite-difference formula to approximate the time derivative $\dot{x}$. For a fixed $\bar{T}$ in Theorem \ref{convergence-saa}, it begins with the division of the time interval $[0,\bar{T}]$ into $N$ subintervals for a fixed step size $h={\bar{T}}/{N}=t_{i+1}-t_{i}$ where $i=0,\cdot\cdot\cdot,N-1$. Inspired by the DVI-specific time-stepping approach in \cite{pangDVI}, we propose to solve the optimization problem \eqref{OCDE-2-SAA} independently of the first equation \eqref{OCDE-1-SAA}. This method is different with the time-stepping method which is usually adopted in \cite{chen2022,luo-chen-2023,luo2021}. Therefore,
 starting from $\mathbf{x}_0^\nu=x_0$, we compute two finite sets of vectors
$
\{\mathbf{x}^{\nu}_{1},\mathbf{x}^{\nu}_{2},\cdot\cdot\cdot,\mathbf{x}^{\nu}_{N}\}\subset \mathbb{R}^n$ and $ \{\mathbf{y}^{\nu}_{1},\mathbf{y}^{\nu}_{2},\cdot\cdot\cdot,\mathbf{y}^{\nu}_{N}\}\subset\mathbb{R}^m
$
in the following manner for $i=0,\cdot\cdot\cdot,N-1$:
\begin{eqnarray}
& &\mathbf{x}_{i+1}^\nu=\mathbf{x}_{i}^\nu+\frac{h}{\nu}\sum_{\ell=1}^\nu f(t_{i+1},\mathbf{x}_{i+1}^\nu,\mathbf{y}_{i+1}^\nu,\xi_\ell), \label{OCDE-1-SAA-time}\\
& &\begin{aligned}
&\mathbf{y}_{i+1}^\nu= \arg\min_{\mathbf{y}\in\mathbb{R}^m} \frac{1}{\nu}\sum_{\ell=1}^\nu {g}(t_{i+1},\mathbf{x}_{i}^\nu,\mathbf{y},\xi_\ell) \label{OCDE-2-SAA-time}&\\
&\quad\quad\quad\quad\textrm{s.t.} \,\,\,\, \mathbf{y}\in K^\nu(t_{i+1},\mathbf{x}_{i}^\nu).\\
 \end{aligned}
\end{eqnarray}

\begin{theorem}\label{existence-saa-time-stepping}
Suppose that Assumption \ref{assumption-convexity} holds and $\textrm{int} K(0,x_0)\neq\emptyset$. Then problem \eqref{OCDE-1-SAA-time}-\eqref{OCDE-2-SAA-time} has a unique solution $\{\mathbf{x}^{\nu}_i,\mathbf{y}^{\nu}_i\}_{i=1}^N$ w.p.1 for sufficiently large $\nu$ and sufficiently small $h$. Moreover, there exists $\hat{\rho}>0$ such that for any $i\in\{0,\cdot\cdot\cdot,N-1\}$,
\begin{equation*}\label{linear-growth-nu-h}
\|\mathbf{y}^{\nu}_{i+1}\|\leq \hat{\rho}(1+\|\mathbf{x}^{\nu}_i\|)
\end{equation*}
holds w.p.1 for sufficiently large $\nu$ and $N$.
\end{theorem}
\begin{proof}
For any  $i\in\{0,\cdot\cdot\cdot,N-1\}$, $\mathbf{y}^\nu_{i+1}$ is a unique optimal solution of problem \eqref{OCDE-2-SAA-time} w.p.1 for sufficiently large $\nu$.  Similar with the proof of Theorem \ref{existence-strong-convex} and for a fixed $t_i$, there exists $\hat{\rho}_i>0$ such that
\begin{equation}\label{linear-growth-nu-h-i}
\|\mathbf{y}^\nu_{i+1}\|\leq \hat{\rho}_i(1+\|\mathbf{x}^\nu_i\|).
\end{equation}

Following from the Lipschitz property of $f(\cdot,\cdot,\cdot,\xi)$ in \eqref{Lipschitz}, we obtain that for any $\tilde{\mathbf{x}}$ and $\bar{\mathbf{x}}\in \mathbb{R}^n$,
\begin{eqnarray*}
&&\left\|\frac{h}{\nu}\sum_{\ell=1}^\nu f(t_{i+1},\tilde{\mathbf{x}},\mathbf{y}^\nu_{i+1},\xi_\ell)-\frac{h}{\nu}\sum_{\ell=1}^\nu f(t_{i+1},\bar{\mathbf{x}},\mathbf{y}^\nu_{i+1},\xi_\ell)\right\|\\
&&\leq\frac{h}{\nu}\sum_{\ell=1}^\nu \|f(t_{i+1},\tilde{\mathbf{x}},\mathbf{y}^\nu_{i+1},\xi_\ell)-
f(t_{i+1},\bar{\mathbf{x}},\mathbf{y}^\nu_{i+1},\xi_\ell)\|
\le \kappa h\|\tilde{\mathbf{x}}-\bar{\mathbf{x}}\|,
\end{eqnarray*}
where $\kappa \ge \mathbb{E}[\kappa_f(\xi)]\ge \frac{1}{\nu}\sum_{\ell=1}^\nu \kappa_f(\xi_\ell)$ w.p.1 for sufficiently large $\nu$. Therefore, if $h<\frac{1}{\kappa}$, we know that $\frac{h}{\nu}\sum_{\ell=1}^\nu f(t_{i+1},\cdot,\mathbf{y}^\nu_{i+1},\xi_\ell)$ is a contractive mapping.
Moreover, there exists $\tilde{\rho}_f>0$ such that for any $i=0,\cdot\cdot\cdot,N-1$
\begin{eqnarray*}
\|\mathbf{x}^\nu_{i+1}\|\leq \|\mathbf{x}^\nu_{i}\|+\frac{h}{\nu}\sum_{\ell=1}^\nu \|f(t_{i+1},\mathbf{x}^\nu_{i+1},\mathbf{y}^\nu_{i+1},\xi_\ell)\|\leq \|\mathbf{x}^\nu_{i}\|+h \tilde{\rho}_f(1+\|\mathbf{x}^\nu_{i+1}\|).
\end{eqnarray*}
It implies that there exists $0<h_0<\frac{1}{\tilde{\rho}_f}$ such that $\|\mathbf{x}^\nu_{i+1}\|\leq \exp(\frac{\tilde{\rho}_f \bar{T}}{1-h_0\tilde{\rho}_f})(1+\|x_0\|)+1$ for $h\in(0,h_0]$.
The contraction mapping theorem implies that there exists unique $\mathbf{x}_{i+1}^\nu$ such that (\ref{OCDE-1-SAA-time}) holds with $i=0,\cdot\cdot\cdot,N-1$. We then conclude that problem \eqref{OCDE-1-SAA-time}-\eqref{OCDE-2-SAA-time} has a unique solution $\{\mathbf{x}^{\nu}_i,\mathbf{y}^{\nu}_i\}_{i=1}^N$ w.p.1 for sufficiently large $\nu$ and sufficiently small $h$ and the linear growth condition \eqref{linear-growth-nu-h-i} holds  by $\hat{\rho}=\max_{i\in\{1,\cdot\cdot\cdot,N\}}\{\hat{\rho}_i\}$.
\end{proof}

Let $\{\mathbf{x}^{\nu}_i,\mathbf{y}^{\nu}_i\}_{i=1}^N$ be a solution of \eqref{OCDE-1-SAA-time}-\eqref{OCDE-2-SAA-time}. We define a piecewise linear function ${x}^{\nu}_h$ and a piecewise constant function ${y}^{\nu}_h$ on $[0,\bar{T}]$ as below:
\begin{eqnarray}\label{feasible-interpolant}
      {x}^{\nu}_h(t) = \mathbf{x}^{\nu}_{i}+\frac{t-t_{i}}{h}(\mathbf{x}^{\nu}_{i+1}
   -\mathbf{x}^{\nu}_{i}), \quad
  {y}^{\nu}_{h}(t) = \mathbf{y}^{\nu}_{i+1},\,\,\,\,\forall\,\, t\in(t_{i},t_{i+1}].
 \end{eqnarray}

\begin{theorem}\label{convergence-saa-time}
Suppose that Assumption \ref{assumption-convexity} holds and $\textrm{int} K(0,x_0)\neq\emptyset$. Let $(x_h^\nu,y_h^\nu)$ be defined in \eqref{feasible-interpolant} associated with a solution $\{\mathbf{x}^{\nu}_i,\mathbf{y}^{\nu}_i\}_{i=1}^N$ of \eqref{OCDE-1-SAA-time}-\eqref{OCDE-2-SAA-time}. Then there are sequences $\{\nu_k\}$ and $\{h_k\}$ with $\nu_k\rightarrow\infty$ and $h_k \downarrow 0$ as $k\rightarrow\infty$,
such that
\[\lim_{\nu_k \rightarrow \infty } \lim_{h_k\downarrow0} \|x^{\nu_k}_{h_k}-x^*\|_s=0,\,\, w.p.1\]
and
  \[\lim_{\nu_k \rightarrow \infty } \lim_{h_k\downarrow0} \|y^{\nu_k}_{h_k}-y^*\|_{L^2}=0,\,\, w.p.1,\]
where $(x^*,y^*)$ is a  weak solution of \eqref{OCDE-1}-\eqref{OCDE-2} over $[0,\bar{T}]$.
\end{theorem}
\begin{proof}
According to Theorems \ref{existence-saa-time-stepping}, we get the family of functions $\{{x}^\nu_h(t)\}$ is uniformly bounded on $[0,\bar{T}]$ w.p.1 for sufficiently large $\nu$ and sufficiently small $h$.
Moreover, for any $\nu>0$,
\begin{eqnarray*}\label{xjie}
  \|\mathbf{x}^\nu_{i+1}-\mathbf{x}^\nu_{i}\| \leq h\tilde{\rho}_f(1+\|\mathbf{x}^\nu_{i+1}\|)\leq  h\tilde{\rho}_f\left(2+\exp(\frac{\tilde{\rho}_f \bar{T}}{1-h_0\tilde{\rho}_f})(1+\|x_0\|)\right)\triangleq h \hat{\alpha}.
\end{eqnarray*}
Then for any $t\in[t_{i},t_{i+1}]$, $\tau\in[t_{i+p},t_{i+p+1}]$, $i\in\{0,\cdot\cdot\cdot,N-1 \}$ and $p\in \{-i,1-i,\cdot\cdot\cdot,N-i-1\}$, we have
\begin{eqnarray*}
\begin{aligned}
  \|{x}^\nu_h(\tau)-{x}^\nu_h(t)\| &= \left\|({x}^\nu_h(\tau)-\mathbf{x}^\nu_{i+p})+\sum_{j=1}^{p-1}(\mathbf{x}^\nu_{i+j+1}-\mathbf{x}^\nu_{i+j})
  +(\mathbf{x}^\nu_{i+1}-{x}^\nu_h(t))\right\| &\\
   &\leq(\tau-t_{i+p}+(p-1)h+t_{i+1}-t)\hat{\alpha} = |\tau-t|\hat{\alpha}.&\\
  \end{aligned}
\end{eqnarray*}
It implies that the piecewise interpolant ${x}^\nu_h$ is Lipschitz continuous on $[0,\bar{T}]$ and the Lipschitz constant is independent of $h$ and $\nu$. Hence we obtain that $\{{x}^\nu_h(t)\}$ is equicontinuous. Then according to the Arzel\'{a}-Ascoli theorem, there are sequences $\{h_k\}$ and $\{\nu_k\}$  with $h_k\downarrow 0 $ and $\nu_k\rightarrow\infty$ as $k\rightarrow\infty$ and an $x^*\in\mathcal{X}_{\bar{T}}$ such that
$\lim_{\nu_k \rightarrow \infty } \lim_{h_k\downarrow0} \|x^{\nu_k}_{h_k}-x^*\|_s=0$ w.p.1.

Let $\hat{y}(t,x(t))$ and $\hat{y}^\nu(t,x(t))$ denote the optimal solutions of optimization problems \eqref{OCDE-2} and \eqref{OCDE-2-SAA} with $t$ and $x(t)$, respectively. For any $t\in(t_i,t_{i+1}]$, it is clear that $\hat{y}^\nu(t_{i+1},\mathbf{x}^\nu_{i})=\mathbf{y}_{i+1}^\nu=y_h^\nu(t)$.
Denote $y^*(t)=\hat{y}(t,x^*(t))$. Then for any $t\in(t_i,t_{i+1}]$ with $i\in\{0,\cdot\cdot\cdot,N-1\}$, we have
\begin{eqnarray*}
\begin{aligned}
\|y_h^\nu(t)-y^*(t)\|\leq& \|\hat{y}^\nu(t_{i+1},\mathbf{x}^\nu_{i})-\hat{y}(t_{i+1},\mathbf{x}^\nu_{i})\|
+\|\hat{y}(t_{i+1},\mathbf{x}^\nu_{i})-\hat{y}(t,x_h^\nu(t))\|&\\
&+\|\hat{y}(t,x_h^\nu(t))-\hat{y}(t,x^*(t))\|.
\end{aligned}
\end{eqnarray*}
From the proof of Theorem \ref{convergence-saa}, we obtain that $\hat{y}^\nu(t,x(t))$ converges to $\hat{y}(t,x(t))$ w.p.1 as $\nu\rightarrow\infty$ for any $t\in \mathbb{R}_+$ and $x(t)\in \mathbb{R}^n$.
When $h\downarrow0$, $t\in(t_i,t_{i+1}]$ and $i\in\{0,\cdot\cdot\cdot,N-1\}$, it is easy to obtain $\|\mathbf{x}^\nu_{i+1}-x_h^\nu(t)\|\rightarrow0$ w.p.1 for sufficiently large $\nu$ from \eqref{feasible-interpolant}. Since $\hat{y}$ is continuous,  with $h_k\downarrow 0$ and $\nu_k\rightarrow\infty$  as $k\rightarrow\infty$ such that $\lim_{\nu_k \rightarrow \infty } \lim_{h_k\downarrow0} \|x^{\nu_k}_{h_k}-x^*\|_s=0$ w.p.1,
 we obtain
\[\lim_{\nu_k \rightarrow \infty } \lim_{h_k\downarrow0} \|y^{\nu_k}_{h_k}-y^*\|_{L^2}=0,\,\, \textrm{w.p.1}.\]

Now we show that $(x^*,y^*)$ is a weak solution of problem \eqref{OCDE-1}-\eqref{OCDE-2} over $[0,\bar{T}]$. For $x_h^\nu(0)=x_0$ and any $t\in(0,\bar{T}]$ (without loss of generality, we assume $t\in(t_i,t_{i+1}]$ with some $i\in\{0,\cdot\cdot\cdot,N-1\}$), we have
\begin{eqnarray*}
\begin{aligned}
& \|\mathcal{W}_h^\nu(t)\|\triangleq\left\|x_h^\nu(t)-x_h^\nu(0)-\int_{0}^t \frac{1}{\nu}\sum_{\ell=1}^\nu f(\tau,x_h^\nu(\tau),y_h^\nu(\tau),\xi_\ell)d\tau\right\|&\\
& \leq \frac{1}{\nu}\sum_{\ell=1}^\nu \kappa_f(\xi_\ell)\left(\sum_{j=0}^{i-1} \int_{t_j}^{t_{j+1}} \|\mathbf{x}^\nu_{j+1}- x_h^\nu(\tau)\|d\tau +\int_{t_i}^{t} \|\mathbf{x}^\nu_{i+1}- x_h^\nu(\tau)\|d\tau \right. &\\
&\left.+\frac{(i+1)h^2}{2}-\frac{(t_{i+1}-t)^2}{2}\right) \leq \frac{h}{\nu}\sum_{\ell=1}^\nu \kappa_f(\xi_\ell)\left(\frac{1}{2}\sum_{j=0}^{i-1}\|\mathbf{x}^\nu_{j+1}- \mathbf{x}^\nu_{j}\|+\|\mathbf{x}^\nu_{i+1}- \mathbf{x}^\nu_{i}\|+\frac{\bar{T}}{2}\right) &\\
&\leq \frac{h (1+\hat{\alpha}) \bar{T}}{\nu}\sum_{\ell=1}^\nu \kappa_f(\xi_\ell),
\end{aligned}
\end{eqnarray*}
where $\kappa_f(\xi)$ is the Lipschitz constant of $f(\cdot,\cdot,\xi)$.
Therefore, we conclude that for any $t\in(0,\bar{T}]$ and two sequences $\{h_k\}$ and $\{\nu_k\}$ with $h_k \downarrow 0$ and $\nu_k \rightarrow\infty$  as $k\rightarrow\infty$,
\[\lim_{\nu_k\rightarrow\infty}\lim_{h_k\downarrow0}\|\mathcal{W}_{h_k}^{\nu_k}\|_s=0,\,\,\, \textrm{w.p.1}.\]

 Obviously,
 \begin{eqnarray*}
\begin{aligned}
&\sup_{t\in[0,\bar{T}]}\left\| x^*(t)-x_0-\int_{0}^t \mathbb{E}[f(\tau,x^*(\tau),y^*(\tau),\xi)]d\tau\right\|&\\
&\leq \lim_{k\rightarrow\infty}\left(\sup_{t\in[0,\bar{T}]}\left\| x^*(t)-x_0-\int_{0}^t \mathbb{E}[f(\tau,x^*(\tau),y^*(\tau),\xi)]d\tau- \mathcal{W}_{h_k}^{\nu_k}(t)\right\|+\|\mathcal{W}_{h_k}^{\nu_k}\|_s\right)&\\
&\leq \lim_{k\rightarrow\infty} \left( \left(1+\frac{\bar{T}}{\nu_k}\sum_{\ell=1}^{\nu_k}\kappa_f(\xi_\ell)\right)\|x^*-x_{h_k}^{\nu_k}\|_s+
\left\| \mathbb{E}[f(\cdot,x^*,y^*,\xi)]
-\frac{1}{\nu_k}\sum_{\ell=1}^{\nu_k}f(\cdot,x^*,y^*,\xi_\ell)\right\|_s \right.&\\
& \left.+\frac{\sqrt{\bar{T}}}{\nu_k}\sum_{\ell=1}^{\nu_k}\kappa_f(\xi_\ell)\|y_{h_k}^{\nu_k}-y^*\|_{L^2}
+\|\mathcal{W}_{h_k}^{\nu_k}\|_s\right)=0,
\end{aligned}
\end{eqnarray*}
which implies that $(x^*,y^*)$ is a weak solution of \eqref{OCDE-1}-\eqref{OCDE-2} over $[0,\bar{T}]$.
\end{proof}

For any fixed $i\in \{1,\ldots, N\}$, solving problem \eqref{OCDE-1-SAA-time}-\eqref{OCDE-2-SAA-time} should address two issues: the nonsmooth fixed point problem and nonsmooth convex optimization problem. For the nonsmooth convex optimization problem \eqref{OCDE-2-SAA-time}, we can adopt the well-known existing algorithms such as proximal schemes. To solve  the nonsmooth fixed point problem \eqref{OCDE-1-SAA-time}, we can adopt the EDIIS algorithm \cite{chen-tim-2019,chen2022} which is a modified Anderson acceleration. The Anderson acceleration is designed to solve the fixed point problem when computing the Jacobian of the function in the problem is impossible or too costly \cite{bianchen2022}. We have known that  $\frac{h}{\nu}\sum_{\ell=1}^\nu f(t_{i+1},\cdot,\mathbf{y}^\nu_{i+1},\xi_\ell)$ is a contractive mapping w.p.1 for sufficiently large $\nu$ and sufficiently small $h$. Then following from \cite[Theorem 2.1]{chen-tim-2019}, we can obtain that the sequence $\{\mathbf{x}^{(\nu, k)}_{i+1}\}$ generated by the EDIIS algorithm converges to the unique solution $\mathbf{x}^\nu_{i+1}$ of \eqref{OCDE-1-SAA-time} as the iteration step $k\rightarrow\infty$.

\section{Numerical experiment}\label{se:numerical-exam}
In this section, we verify our theoretical results by a numerical example, which is performed in MATLAB 2017b on a Lenovo laptop (2.60GHz, 32.0GB RAM).

\textit{Example 6.1.} We consider the following problem:
\begin{equation}\label{example}
\begin{aligned}
& \dot{x}(t)=\mathbb{E}\left[\left(
           \begin{array}{cc}
             \xi_1 & 2\\
             \xi_1^2 & \xi_2 \\
           \end{array}
         \right){x}(t)+\left(
           \begin{array}{ccc}
             2x_1(t) & x_2(t) & \xi_2 \\
             2t & 0 & \xi_1x_1(t) \\
           \end{array}
         \right)y(t)\right],&\\
& y(t)\in \arg\min_{\mathbf{y}\in \mathbb{R}^3}\mathbb{E}[ \|M(\xi)\mathbf{y}-b(x(t),\xi)\|^2+\|\mathbf{y}\|_1] &\\
&\quad\quad\quad \textrm{s.t.} \,\,\mathbf{y}\in K(t,x(t))=\{\mathbf{y} :\mathbb{E}[A(\xi)]x(t)+\mathbb{E}[B(\xi)]\mathbf{y}+\mathbb{E}[Q(t,\xi)]\leq 0\},
\end{aligned}
\end{equation}
where $x(t)=(x_1(t), x_2(t))^\top$, $x(0)=x_0=(-1,-2)^\top$,
\begin{equation*}
\begin{aligned}
& M(\xi)=\left( \begin{array}{ccc}
  2+\xi_1 & 0 & -\xi_2  \\
  0 & \xi_1+\xi_2 & -1  \\
  \end{array}\right), \quad  \quad b(x(t),\xi)
       =\left(\begin{array}{c}
         x_1(t) +\xi_2 \\
         \xi_1 x_2(t)\\
       \end{array}\right),\\
 &   A(\xi)=\left( \begin{array}{cc}
  -2-\xi_1 & 1  \\
  -1 & \xi_2  \\
  \end{array}\right),\,\, B(\xi)=\left(
   \begin{array}{ccc}
   1 & \xi_1^2 & \xi_2 \\
   \xi_2 & 0 & 2 \\
   \end{array}
   \right),\,\, Q(t,\xi)=\left( \begin{array}{c}
            t-\xi_1 \\
               \xi_2 \\
             \end{array}
                \right).
\end{aligned}
\end{equation*}
We set the terminal time $T=1$, $\xi_1\sim {\cal{N}}(1,0.01)$ and $\xi_2\sim
{\cal{U}}(-1,1)$. It can be verified easily that all functions in this example fulfill our settings in the beginning of this paper.
It is obvious that the objective function $\mathbb{E}[g(t,x(t),\cdot)]$ in \eqref{example} is convex and $\mathbb{E}[g(0,x_0,\mathbf{y})]\geq \|\mathbf{y}\|_1$, which means that $\mathbb{E}[g(0,x_0,\cdot)]$ is level-coercive and then is level-bounded, following from \cite[Corollary 3.27]{Rockafellar}.
We can also obtain that $ K(0,x_0)=\{(\mathbf{y}_1,\mathbf{y}_2,\mathbf{y}_3):\mathbf{y}_1+1.01\mathbf{y}_2\leq0, 2\mathbf{y}_3+1\leq0\}$ and then $\textrm{int} K(0,x_0)\neq\emptyset$. Hence, we know that Assumption \ref{assumption-interior-point} holds for this example.

Now we illustrate that this example exists a solution on $[0,1]$. Following from the proof of Theorem \ref{existence-weak}, the solution existing interval mainly depends on the range of $(t,\mathbf{x})$ which is such that $\textrm{int} K(t,\mathbf{x})\neq\emptyset$ and $\mathbb{E}[g(t,\mathbf{x},\cdot)]$ is level-bounded.
By $ K(t,\mathbf{x})=\{(\mathbf{y}_1,\mathbf{y}_2,\mathbf{y}_3):\mathbf{y}_1+1.01\mathbf{y}_2\leq3\mathbf{x}_1-\mathbf{x}_2-t+1, 2\mathbf{y}_3\leq -\mathbf{x}_1\}$, we know that for any $\mathbf{x}=(\mathbf{x}_1,\mathbf{x}_2)$ there always holds that $\textrm{int} K(t,\mathbf{x})\neq\emptyset$ because two linear constraints are independent of each other. For any $(t,\mathbf{x})$, we can also have $\mathbb{E}[g(t,\mathbf{x},\mathbf{y})]\geq \|\mathbf{y}\|_1$, which means that $\mathbb{E}[g(t,\mathbf{x},\cdot)]$ is also level-coercive and then is level-bounded. Hence, we know that the optimal set $\mathcal{S}(t,\mathbf{x})$ of the optimization problem in \eqref{example} is nonempty and bounded for any $(t,\mathbf{x})$, and is upper semicontinuous. It then derives that \eqref{example} has at least a weak solution on $[0,1]$ from Theorem \ref{existence-weak}.

We add  a regularization term $\mu\|\mathbf{y}\|^2$ to the objective function in \eqref{example}.
For the regularization numerical form of \eqref{example}, we use the EDIIS(1) method to solve the fixed point problem and the Matlab toolbox CVX to obtain the optimal solution of the convex optimization problem. The EDIIS(1) method is used in the numerical example of \cite{chen2022} since the minimization problem in EDIIS has a closed-form solution in this case. The stop criterion of EDIIS(1) for each $i$ is $\|\mathbf{x}^{(\nu,k+1)}-\mathbf{x}^{(\nu,k)}\|\leq10^{-6}$.

In the numerical experiments, let $\hat{x}=(\hat{x}_1,\hat{x}_2)$ be a numerical solution of the ODE in \eqref{example} with regularization parameter $\mu=10^{-5}$, sample size $\nu=5000$ and step size $h=10^{-4}$. For the fixed step size $h=10^{-4}$, we carry out tests with the regularization parameter $\mu=10^{-4}$, $0.001$, $0.01$ and $0.1$, and the sample size $\nu=3000$, $2000$, $1000$ and $500$. We compute the numerical solution $x^{\mu,\nu}=(x_1^{\mu,\nu},x_2^{\mu,\nu})$ and
\[R_1=\frac{1}{10000}\sum_{i=1}^{10000}|\hat{x}_1(ih)-x_1^{\mu,\nu}(ih)|,\,\,
R_2=\frac{1}{10000}\sum_{i=1}^{10000}|\hat{x}_2(ih)-x_2^{\mu,\nu}(ih)|\]
50 times and averages them. The decreasing tendencies of $R_1$ and $R_2$ as $\nu$ increases and $\mu$ decreases are shown in FIG. \ref{figure0numericalex}.

\begin{figure}[htbp]
  \centerline{
\subfigure{\psfig{figure=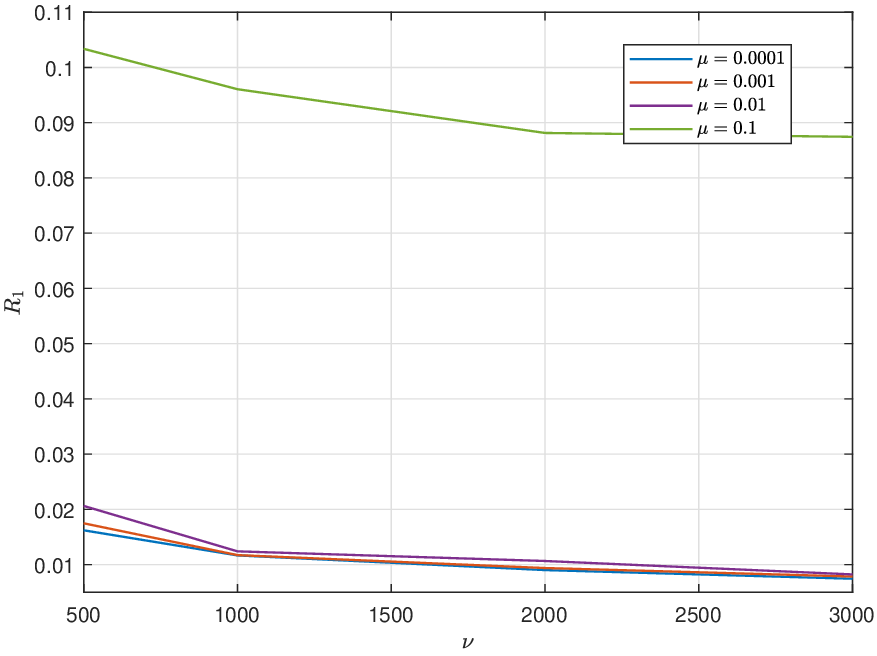,width=0.45\textwidth}}
\subfigure{\psfig{figure=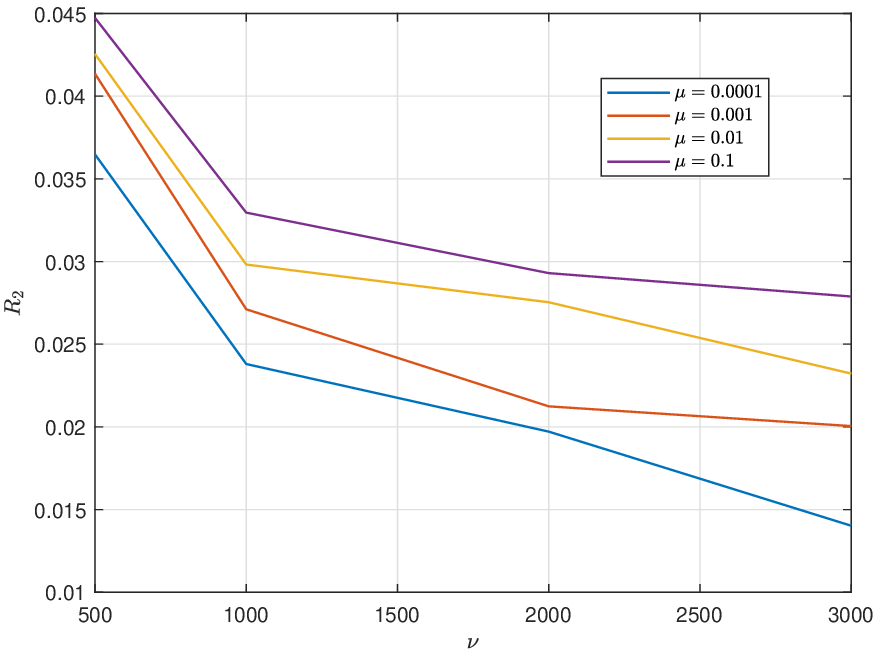,width=0.45\textwidth}}
}\caption{ The decreasing tendencies of $R_1$ and $R_2$ as $\nu$ increases and $\mu$ decreases.}\label{figure0numericalex}
\end{figure}

\section{Application in time-varying parameter estimation for an ODE}\label{se:application}

In this section, we apply \eqref{OCDE-1}-\eqref{OCDE-2} to estimate the time-varying parameter for an ODE. Several strategies can be employed to estimate the time-varying parameters for an ODE based on noisy data, such as the local polynomial method \cite{chen-wu-2008}, the nonlinear least squares method \cite{lin-ying-2001} and the spline-based method \cite{liang-wu-2008}. As a Bayesian approach, Gaussian process is also widely used to infer dynamics of ODE (see \cite{yangpans2021} and the references therein).
A Gaussian process can be viewed as a distribution over functions, while its inference takes place directly in the function space. It is a collection of random variables, any finite number of which have a joint Gaussian distribution. It is also a non-parametric probabilistic model for function estimation that is widely used in tasks such as regression and classification. Therefore, we  use Gaussian process to infer the dynamics of an ODE based on noisy data in order to estimate its time-varying parameter.

A system of ODE with initial value $x(0)=x_0$ takes the form
\begin{equation}\label{GP-ODE-example}
\dot{x}(t)=f_1(x(t))y(t)+f_2(x(t)),\,\,\,t\in[0,T],
\end{equation}
where $f_1:\mathbb{R}^{n}\rightarrow\mathbb{R}^{n\times m}$, $f_2:\mathbb{R}^{n}\rightarrow\mathbb{R}^{n}$ are given functions and $y(t)\in \mathbb{R}^{m}$ is unknown.
It is well known that a Gaussian process is completely characterized by a mean function and a covariance or a kernel function.

We assume that we can observe the values of states and their derivatives at the given time points $\{t_i\}_{i=1}^{N}$. If the observation data of derivatives is not available, we can estimate the derivatives by the first-order method, that is $\dot{x}_j(t_i)\approx \frac{x(t_{i+1})-x(t_i)}{t_{i+1}-t_i}$.
Let $Y_i=[Y^1_i,\cdot\cdot\cdot,Y_i^{n}]^\top$ be the measurement of true value of state variable $x$ at time $t_i$, that is $Y^j_i=x_j(t_i)+\epsilon_{j}$ for $j=1,\cdot\cdot\cdot,n$, where $\epsilon_{j}$ denotes the measurement error. We also let $Z^j_i=\dot{x}_j(t_i)+\epsilon^d_{j}$ for given derivatives observation or $Z^j_i=\frac{Y^j_{i+1}-Y^j_{i}}{t_{i+1}-t_i}$, where $\epsilon^d_{j}$ denotes the measurement error. The errors $\epsilon_{j}$ and $\epsilon^d_{j}$ are assumed to follow a Gaussian distribution with zero mean and variance $\sigma_j^2$ and $\hat{\sigma}_j^2$, respectively.

We employ the Gaussian processes to obtain the distributions of the state variables and their derivatives, denoted as $\hat{x}(t,\xi)$ and $\dot{\hat{x}}(t,\xi)$, where $\xi$ denotes a random variable. It should be noticed that $\hat{x}(t,\xi)$ and $\dot{\hat{x}}(t,\xi)$ have no closed forms but can obtain their values at any given $t$.
By \cite{ODIN-2020}, the $n$-dimensional variable of $\hat{x}(t,\xi)$ and $\dot{\hat{x}}(t,\xi)$ can be obtained by stacking $n$ independent Gaussian processes to model each state and the derivative independently.

 Therefore, we can estimate the time-varying coefficients $y(t)$ by solving the following optimization problem
\begin{equation}\label{GP-example}
\begin{aligned}
y(t)\in \arg\min_{\mathbf{y}\in \mathbb{K}} \mathbb{E}\left[\|\dot{\hat{x}}(t,\xi)-f_1(x(t))\mathbf{y}-f_2(x(t))\|_1
+\|\hat{x}(t,\xi)-x(t)\|_1\right],
\end{aligned}
\end{equation}
where $x(t)$ fulfills \eqref{GP-ODE-example}, and $\mathbb{K}$ is a nonempty closed convex set which can be some inaccurate information for the coefficients such as upper or lower bounds. Obviously, the objective function in \eqref{GP-example} is nonsmooth and convex in $\mathbf{y}$. Note that the objective function is not strongly convex in $\mathbf{y}$, then we introduce the regularization method into it.
By using the regularization method with parameter $\mu$, SAA with sample size $\nu$ and time-stepping method with step size $h$, we obtain the following discrete form of \eqref{GP-ODE-example}-\eqref{GP-example}:
\begin{eqnarray}
& &\mathbf{x}_{i+1}=\mathbf{x}_{i}+h(f_1(\mathbf{x}_{i+1})\mathbf{y}_{i+1}+f_2(\mathbf{x}_{i+1})), \label{GP-ODE-SAA-time}\\
& &\begin{aligned}
&\mathbf{y}_{i+1}= \arg\min_{\mathbf{y}\in\mathbb{K}} \frac{1}{\nu}\sum_{\ell=1}^\nu \|\dot{\hat{x}}(t_{i+1},\xi_\ell)-f_1(\mathbf{x}_i)\mathbf{y}-f_2(\mathbf{x}_i)\|_1+\mu\|\mathbf{y}\|^2. \label{GP-SAA-time}&\\
 \end{aligned}
\end{eqnarray}
It should be noted that $\dot{\hat{x}}(t_{i+1},\xi_\ell)$ does not need any information of $\mathbf{x}_{i}$ and $\mathbf{x}_{i+1}$ in \eqref{GP-SAA-time}, as it is obtained by the Gaussian process based on the observation data independently.
As we mentioned before, we adopt EDIIS algorithm to solve the fixed point problem \eqref{GP-ODE-SAA-time}. For the nonsmooth convex optimization problem \eqref{GP-SAA-time}, we use the CVX tool box to solve it. At last, we can obtain the approximation solution of \eqref{GP-ODE-example}-\eqref{GP-example} and the estimation of time-varying coefficients.

\textit{Example 6.2.}
For the following ODE with time-varying coefficients,
\begin{equation}\label{ex-ode}
\begin{aligned}
&\dot{x}_1(t)=x_1(t)+\sin(t)x_1(t), &\\
&\dot{x}_2(t)=x_1(t)-2tx_2(t), \,\,\,t\in[0,5],
\end{aligned}
\end{equation}
where $x_1(0)=1$ and $x_2(0)=0$.
Let $t_i=0.04i$, $i=0,\cdot\cdot\cdot,125$. Obviously, for any given $t_i$, we can obtain the values $x(t_i)$ and then $\dot{x}(t_i)$. We estimate the parameters $(\sin(t),-2t)$ of \eqref{ex-ode} under two cases: (i) both the noisy data of $Y_i={x}(t_i)+\epsilon$ and $Z_i=\dot{x}(t_i)+\epsilon$ are given, where $\epsilon\sim {\cal{N}}(0,0.4)$; (ii) only the noisy data $Y_i$ is given.

We estimate the time-varying parameters of \eqref{ex-ode} by solving the problem \eqref{ex-ode} with an optimization problem \eqref{GP-example}, where we estimate the parameters with the set $\mathbb{K}=\{(\mathbf{y}_1,\mathbf{y}_2):-1\leq \mathbf{y}_1\leq 1, -10\leq \mathbf{y}_2\leq 0\}$.
When we adopt the regularization approach, SAA method and the time-stepping method, we set the regularization parameter $\mu=10^{-4}$, the sample size $\nu=1000$ and step size $h=0.001$. For problem \eqref{GP-ODE-SAA-time}, we also adopt EDIIS(1), where the stop criterion for each $i$ is $\|\mathbf{x}^{(k+1)}_i-\mathbf{x}^{(k)}_i\|\leq10^{-6}$.
We obtain the estimation of parameters by averaging 50 independent repetitions.
The visualization of estimates of parameters $(\sin(t),-2t)$ in \eqref{ex-ode} for the two cases are shown in FIGs. \ref{figure1} and \ref{figure2}. Let $y(t)$ and $\tilde{y}(t)$ denote the true functions and their estimations, respectively. In the figures, the dash lines denote the $95\%$ simultaneous $l_\infty$ credible bands, where the radius is estimated by the $95\%$ quantile of $\|y-\tilde{y}\|_s\triangleq \max_{i}|y(t_i)-\tilde{y}(t_i)|$. The shaded area denotes the estimation area between the $25\%$ and $75\%$ quantiles of 50 independent repetitions.

\begin{figure}[htbp]
  \centerline{
\subfigure{\psfig{figure=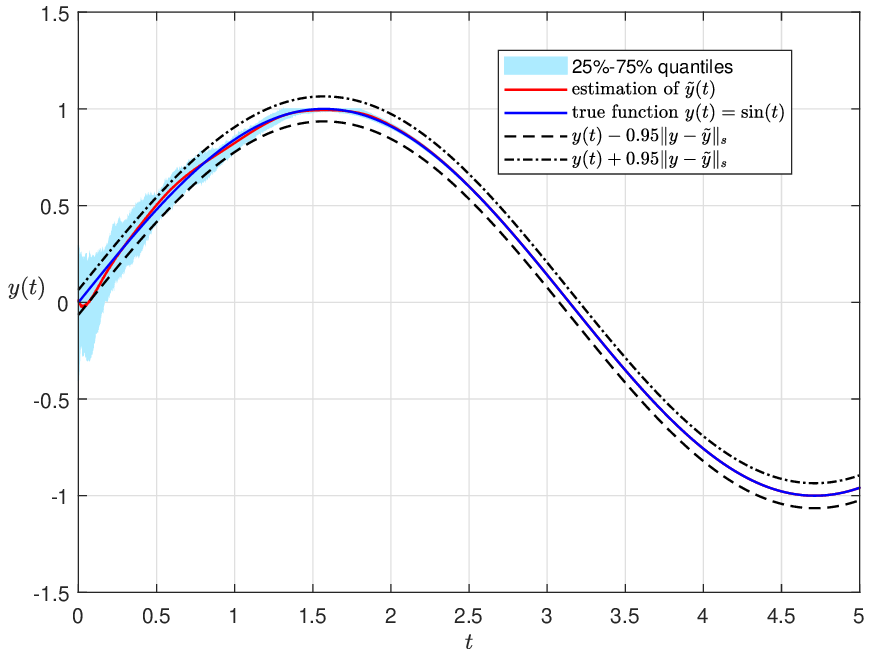,width=0.45\textwidth}}
\subfigure{\psfig{figure=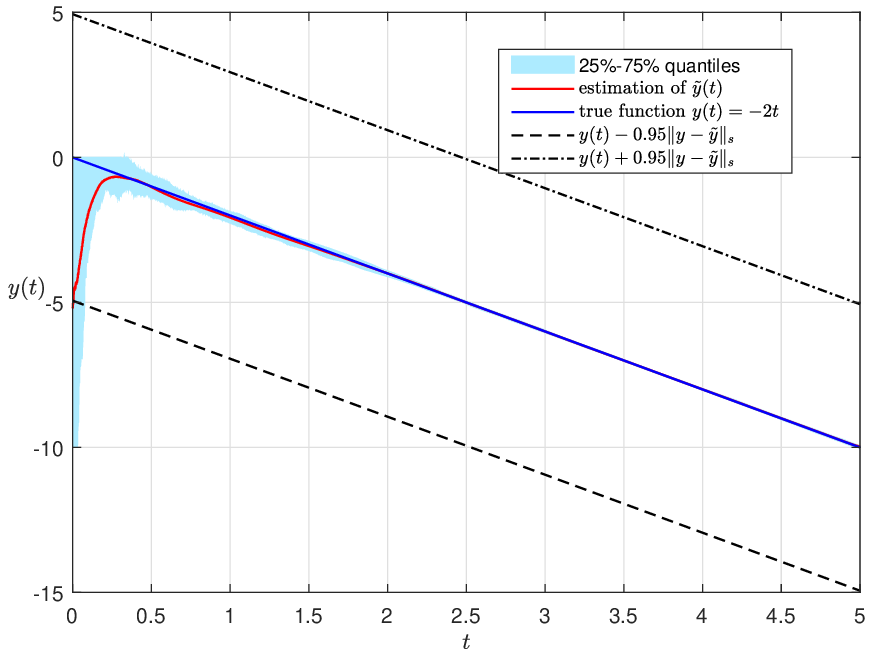,width=0.45\textwidth}}
}\caption{ Visualization of estimates of parameters $(\sin(t),-2t)$ in \eqref{ex-ode} under case (i).}\label{figure1}
\end{figure}

\begin{figure}[htbp]
  \centerline{
\subfigure{\psfig{figure=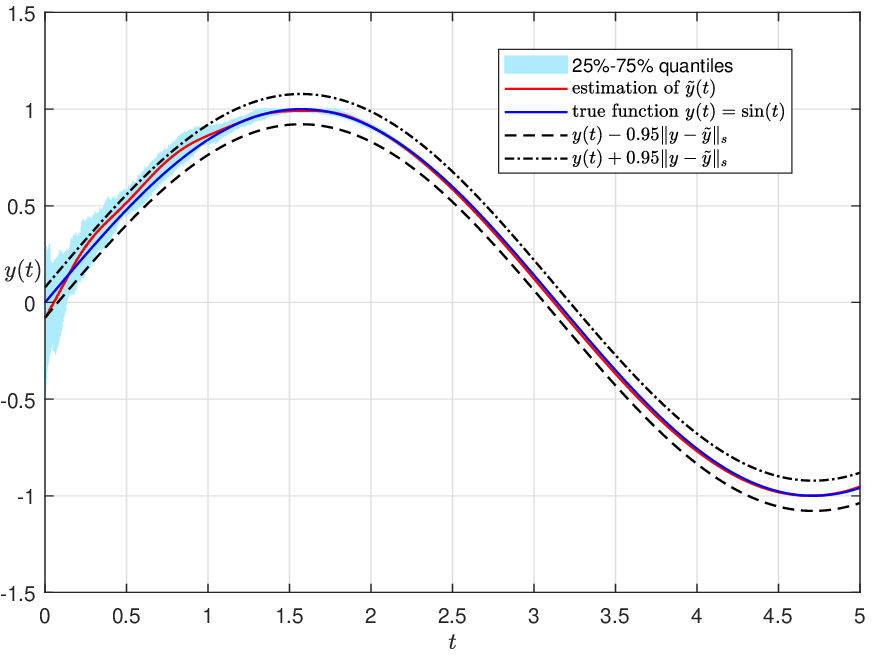,width=0.45\textwidth}}
\subfigure{\psfig{figure=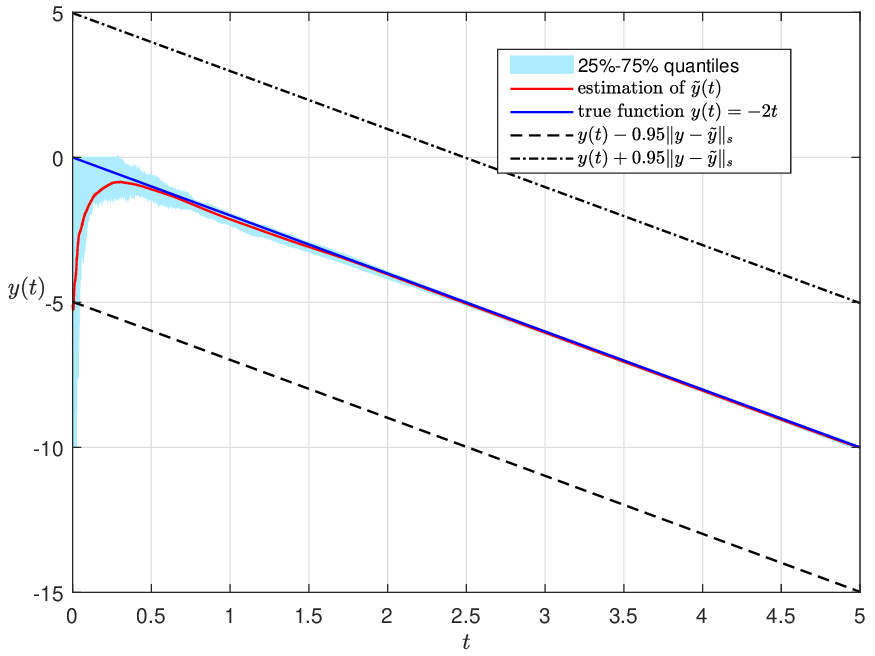,width=0.45\textwidth}}
}\caption{ Visualization of estimates of parameters $(\sin(t),-2t)$ in \eqref{ex-ode} under case (ii).}\label{figure2}
\end{figure}

Under the case (i), for the fixed step size $h=0.001$, we also carry out tests with the regularization parameter $\mu=0.001$, $0.01$, $0.1$ and $1$, and the sample size $\nu=1000$, $500$, $100$ and $50$. We compute the numerical solution $x^{\mu,\nu}=(x_1^{\mu,\nu},x_2^{\mu,\nu})$ and
\[R_3=\frac{1}{5000}\sqrt{\sum_{i=1}^{5000}({x}^*_1(ih)-x_1^{\mu,\nu}(ih))^2},\,\,
R_4=\frac{1}{5000}\sqrt{\sum_{i=1}^{5000}({x}^*_2(ih)-x_2^{\mu,\nu}(ih))^2}\]
50 times and averages them, where $(x_1^*(t),x_2^*(t))$ is the true solution of problem \eqref{ex-ode}
\[x_1^*(t)=e^{t-\cos(t)+1},\,\,\,\, x_2^*(t)=e^{1-t^2}\int_0^t{e^{-\tau^2-\cos(\tau)+\tau}}d\tau.\]
  The decreasing tendencies of $R_3$ and $R_4$ as $\nu$ increases and $\mu$ decreases are shown in FIG. \ref{figure3}.

\begin{figure}[htbp]
  \centerline{
\subfigure{\psfig{figure=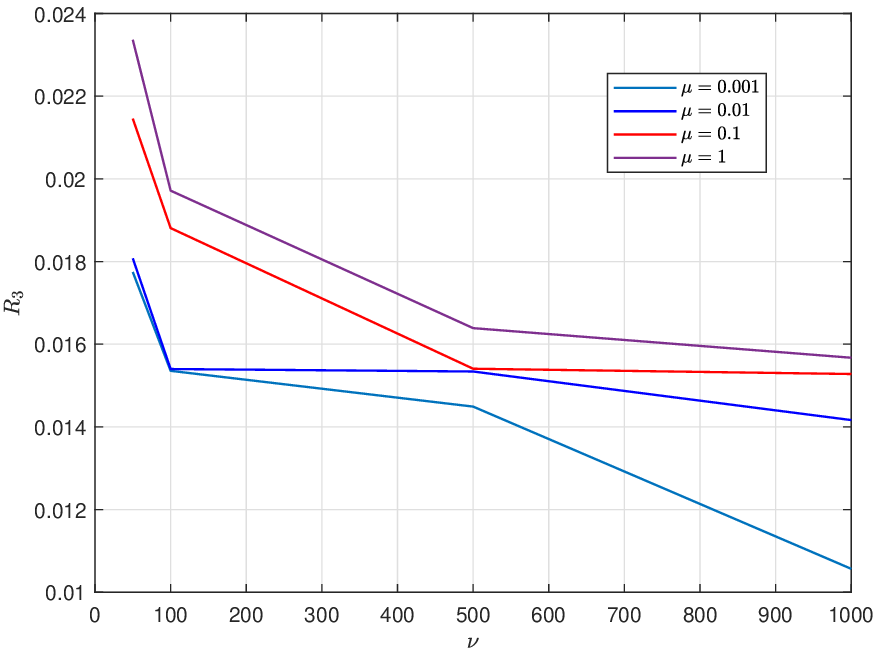,width=0.45\textwidth}}
\subfigure{\psfig{figure=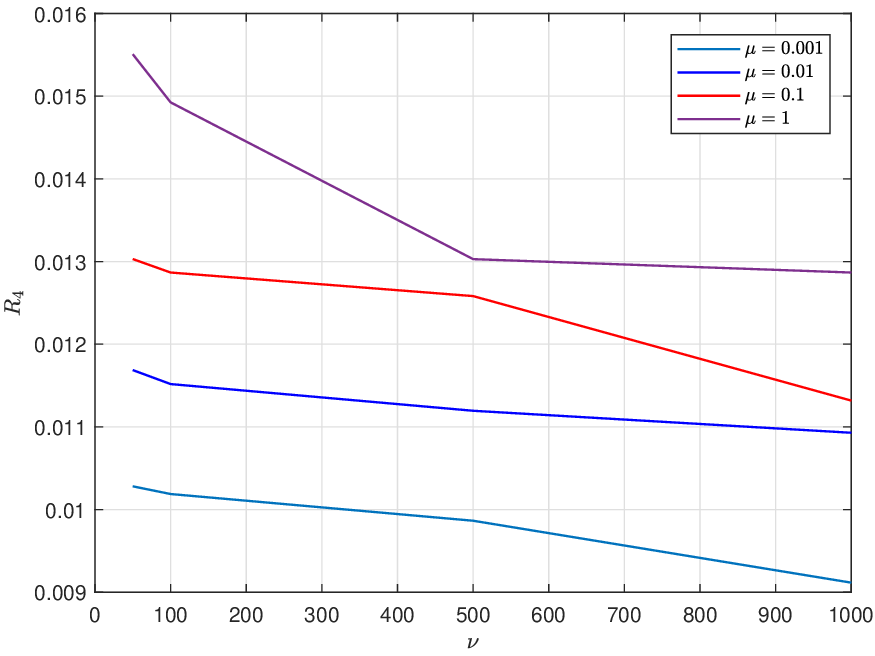,width=0.45\textwidth}}
}\caption{ The decreasing tendencies of $R_3$ and $R_4$ as $\nu$ increases and $\mu$ decreases for \eqref{ex-ode} under case (i).}\label{figure3}
\end{figure}

From FIGs \ref{figure1} and \ref{figure2}, we can observe that our model \eqref{GP-ODE-example}-\eqref{GP-example} can be applied to approximate the time-varying parameters in an ODE system \eqref{ex-ode}, which means the potential application of (\ref{OCDE-1})-(\ref{OCDE-2}) in estimating the time-varying parameters in ODE system. FIG \ref{figure3} also verifies the theoretical results for our numerical methods proposed by this paper.

\section{Conclusions}\label{se:conclusions}

In this paper, we show the existence of weak solutions of the dynamic system coupled with solutions of stochastic nonsmooth convex optimization problem (\ref{OCDE-1})-(\ref{OCDE-2}).  By adding a regularization term $\mu\|\mathbf{y}\|^2$ to the convex objective function in (\ref{OCDE-2}), the convex optimization problem becomes a strongly convex problem, which has a unique continuous optimal solution. We show that the unique optimal solution of nonsmooth optimization with strong convexity admits a linear growth condition and the regularized dynamic system has a classic solution. Moreover,  we prove that the solutions of regularized problem converge to the solutions of original problem as the regularization parameter goes to zero.
Moreover, we show that the unique optimal solution of the regularized optimization problem (\ref{OCDE-2-regularization})  converges to the least-norm optimal solution of the original problem (\ref{OCDE-2}). We adopt the sample average approximation scheme and implicit Euler method to discretize the  dynamic system coupled with solutions of stochastic nonsmooth strongly convex optimization problem  and present the corresponding convergence analysis. We give a numerical example to demonstrate our theoretical results. Finally, the effectiveness of our model is  verified by an example of the estimation of the time-varying parameters in ODE.\\

{\bf Acknowledgement}  We would like to thank the Associate Editor and two referees
for their very helpful comments.

\bibliographystyle{siamplain}
\bibliography{references}

\end{document}